\documentclass[3p]{elsarticle}

\usepackage[utf8]{inputenc}
\usepackage{amsmath}
\usepackage{amssymb}
\usepackage{mathtools}
\usepackage{amsthm}
\usepackage{hyperref}

\usepackage[dvipsnames]{xcolor}

\newtheorem{theorem}{Theorem}[section]

\newtheorem{lemma}{Lemma}[theorem]
\newtheorem{proposition}{Proposition}[section]
\newtheorem{remark}{Remark}[theorem]

\DeclareMathOperator{\minmod}{minmod}
\DeclareMathOperator{\lcm}{lcm}
\DeclareMathOperator*{\argmax}{arg\,max}

\newcommand{\half}{\frac{1}{2}}
\usepackage{numcompress}\biboptions{numbers,sort&compress}

\begin{document}

\begin{frontmatter}
	
\title{A Finite Volume Method for Continuum Limit Equations of Nonlocally Interacting Active Chiral Particles}
\author[1]{Nikita Kruk}
\ead{nikita.kruk@bcs.tu-darmstadt.de}
\author[2]{Jos\'{e} A. Carrillo}
\ead{carrillo@maths.ox.ac.uk}
\author[1]{Heinz Koeppl\corref{cor1}}
\ead{heinz.koeppl@bcs.tu-darmstadt.de}
\cortext[cor1]{Corresponding author}
\address[1]{Department of Electrical Engineering and Information Technology, Technische Universit\"{a}t Darmstadt, Rundeturmstrasse 12, 64283, Darmstadt, Germany}
\address[2]{Mathematical Institute, University of Oxford, Oxford OX2 6GG, UK}
	
\begin{abstract}
	The continuum description of active particle systems is an efficient instrument to analyze a finite size particle dynamics in the limit of a large number of particles. However, it is often the case that such equations appear as nonlinear integro-differential equations and purely analytical treatment becomes quite limited. We propose a general framework of finite volume methods (FVMs) to numerically solve partial differential equations (PDEs) of the continuum limit of nonlocally interacting chiral active particle systems confined to two dimensions. We demonstrate the performance of the method on spatially homogeneous problems, where the comparison to analytical results is available, and on general spatially inhomogeneous equations, where pattern formation is predicted by kinetic theory. We numerically investigate phase transitions of particular problems in both spatially homogeneous and inhomogeneous regimes and report the existence of different first and second order transitions.
\end{abstract}


\end{frontmatter}

\section{Introduction}

Collective motion of groups of multiple agents of various origin is a fascinating phenomenon that manifests itself in versatile systems, ranging from microscopic scale, e.g., colloidal suspensions \cite{palacci2013:science,geyer2018:nature_materials,kaiser2017:science}, through mesoscopic scale, e.g., bacterial suspensions and microtubule bundles \cite{thampi2016:science,chen2017:nature,wu2017:science}, to macroscopic scale, e.g., schools of fish and flocks of birds \cite{lopez2012:interface,cavagna2014:arcmp,DDT14}. The study of such systems is an active field of research, both theoretically and experimentally. When a system consists of self-propelled agents, i.e., the ones that are able to move without any external forcing, it is referred to as an active matter system. The direct approach to model such systems is to describe the motion of each constituent agent with differential equations. However, when the number of agents is large, this approach becomes computationally expensive and one looks for a respective continuum description.

First approaches for the continuum description of active matter were based on symmetry arguments and conservation laws \cite{toner2005:ap,bertin2006:pre} and described the evolution of hydrodynamic variables, e.g., a marginal density function and a polar order field. This approach allowed to reproduce some of the behavior given by an agent-based description but the resulting equations were not linked to the microscopic parameters, thus, not allowing for the respective analysis of the agent-based behavior. As an alternative, the kinetic theory presents a systematic way to construct a continuum description of an agent system via density functions of agent's positions and velocities \cite{CDP09,CFTV10,canizo:mmmas,CCH14,carrillo:jsp,lancellotti:ttsp,neunzert,archer:jpa}. Mostly, one is interested in a time evolution of a density function of one agent, which is governed by a nonlinear PDE. As a result, the terms in such a PDE do depend on microscopic parameters from the agent-based model, thus, allowing one to study the continuum limit behavior of a particle system in terms of those parameters.

In this paper, we are interested in the construction of effective FVMs for numerical integration of PDEs derived as the continuum limit of nonlocally interacting self-propelled particle systems in two dimensions that takes the following form:
\begin{equation}
\label{eq:continuum_agent_description}
	\partial_t f(r,\varphi,t) = -v_0e(\varphi) \cdot \nabla_r f(r,\varphi,t) - \partial_{\varphi}\Bigl[w[f](r,\varphi,t) f(r,\varphi,t)\Bigr] + D_\varphi\partial_{\varphi\varphi}f(r,\varphi,t),
\end{equation}
where $r\in\mathbb{R}^2$ is a position vector, $e(\varphi)\in\mathbb{S}^1$ is a unit velocity vector that depends on particle's orientation $\varphi \in \mathbb{R}/(2\pi\mathbb{Z})=:\mathbb{T}$, $t \in \mathbb{R}_+$ is time, $D_\varphi \in \mathbb{R}_+$ is a rotational diffusion coefficient, $f: \mathbb{R}^2\times\mathbb{T}\times\mathbb{R}_+ \rightarrow \mathbb{R}_+$ is a one-particle probability density function which quantifies the probability to find a particle with a position $r$, orientation $\varphi$ at time $t$, $w$ is some functional that represents nonlocal interaction between particles. Besides the high dimensionality of the problem, the presence of the nonlocal interaction requires particular attention to the performance of constructed numerical schemes. Numerical integration of such continuum limit kinetic equations for active matter systems is an ongoing research \cite{carrillo:commun_comput_phys,thueroff:prx2014,poeschel_schwager,ihle2013:pre,carrillo2016:jcp,BCCD16}. For a general survey on the numerical treatment of kinetic equations, we refer the reader to \cite{dimarco_pareschi_2014}.

An important problem in active matter systems is the study of related phase transitions versus model parameters. Its knowledge allows one to analyze, predict, control, and design particle systems with experimentally desirable properties. Depending on the context, one might consider particles of different origins, which in the continuum limit frequently assume polar or nematic representation with polar or nematic interactions \cite{bertin2009:jpamt,degond2015:arma,peshkov2012:prl,peshkov2014:epjst,nagai2015:prl,BCCD16,patelli2019:prl,levis2019:prr,carrillo2019:jcp,frouvelle2012siam_jma}. Phase transitions are commonly quantified in terms of a polar order parameter. It has been shown that depending on different regimes, one might observe first as well as second order transitions.

Most active matter models describe particles that propel themselves along the direction of motion. Besides such models, there has recently been an increase of interest in models where particles can undertake circular motion, known as chiral active matter. Its examples include bacterial swarming close to boundaries of a substrate \cite{lauga2006bpj,lemelle2010jb}, irregular vortex structures in dense suspensions of swimming bacteria \cite{sumino2012nature}, swarming of magnetotactic bacteria in a rotating magnetic field \cite{erglis2007bpj,cebers2011jmmm}, swimming of sperm cells \cite{riedel2005science,friedrich2007pnas}, and shimmering behavior of giant honeybees against predatory wasps \cite{kastberger2008plos_one}. In the present work, we are interested in two-dimensional stochastic dynamics of self-propelled particles with alignment interactions subject to a phase lag. The presence of the phase lag induces particle rotation, rendering the model chiral. Unlike other models, where such rotation is explicitly regulated with a given rotational frequency \cite{DDT14,liebchen2017prl,chen2017:nature}, our model describes irrotational particles that rotate only upon interactions, thus, exhibiting spontaneous symmetry breaking as a collective phenomenon. It appears that particles with such shifted alignment interactions may self-organize into a large variety of spatially inhomogeneous chiral patterns \cite{kruk2020pre} such as traveling bands, dense clouds, and vortices as well as localized chimera states \cite{kruk:aps2018}. In \cite{kruk2020pre}, we presented collective dynamics in terms of a finite-size particle model and used its kinetic and hydrodynamics descriptions to perform stability analysis of spatially homogeneous analytic solutions. In this paper, we are interested in numerical analysis of the kinetic description for spatially inhomogeneous dynamics, which are not expressed analytically. Moreover, by developing an accurate numerical scheme, we want to analyse the nature of phase transitions between different states of the model of interest.

The paper is organized as follows. In Section~\ref{sec:continuum_model}, we introduce a particular particle model and its continuum limit PDE, which we are interested to numerically investigate. We also provide some of the analytical results that will be used later in the construction of numerical methods as well as in their performance assessment. In Section~\ref{sec:numerical_scheme}, we describe the construction of FVMs appropriate for numerical investigation of active matter continuum limit equations in the form of Eq.~\eqref{eq:continuum_agent_description}. In Section~\ref{sec:numerical_tests}, we demonstrate how our schemes perform in recovering spatially homogeneous and inhomogeneous solutions as well as study related phase transitions. Finally, we summarize results of the present paper and outline the future work in Section~\ref{sec:conclusion}.

\section{Active Brownian particle flow in the continuum limit\label{sec:continuum_model}}

We begin the discussion by first introducing an exemplary particle model that lies in the origins of the current research. The model is a starting point for the definition of continuum limit equations \eqref{eq:continuum_agent_description}, which we want to numerically integrate. Moreover, it serves as a reference point for interpretation of solutions of continuum limit PDEs. Next, we will present a PDE corresponding to the particle model along with its key properties, which will be used to test the performance of finite volume schemes in Section~\eqref{sec:numerical_tests}.

\subsection{Finite size particle model}

Let $\mathbb{U} \coloneqq \mathbb{R}/(L\mathbb{Z})$ and $\mathbb{T} \coloneqq \mathbb{R}/(2\pi\mathbb{Z})$ be one-dimensional spaces with periodic boundaries extending from $[0,L]$ and $[0,2\pi]$, respectively. We consider a system of $N\in\mathbb{N}$ self-propelled particles moving in a two-dimensional domain with periodic boundaries. Each particle is assigned a unique index $i=1,\dots,N$. A spatial position of a particle is described with a vector with periodic components $r_i=(x_i,y_i)\in\mathbb{U}^2$. We assume that particles move in an overdamped regime with a constant speed $v_0\in\mathbb{R}_+$. Thereby, their velocity is completely determined by their direction of motion $\varphi_i\in\mathbb{T}$. Let us denote a unit velocity vector as $e(\varphi_i)=(\cos\varphi_i,\sin\varphi_i)\in\mathbb{S}^1$. As a result, we consider the state vector of a particle $i$ to be $p_i=(x_i,y_i,\varphi_i)\in\mathbb{U}^2\times\mathbb{T} =:\Omega$. We consider particles as interacting $\Omega$-valued processes that are solutions to the following coupled system of stochastic differential equations (SDEs):
\begin{equation}
\label{eq:chimera_sde}
	\mathrm{d}r_i = v_0 e(\varphi_i)\; \mathrm{d}t,\qquad \mathrm{d}\varphi_i = \frac{\sigma}{\vert B_{\varrho}^i \vert} \sum_{j \in B_{\varrho}^i} \sin(\varphi_j - \varphi_i - \alpha)\; \mathrm{d}t + \sqrt{2D_\varphi}\; \mathrm{d}W_i.
\end{equation}
The first equation described a self-propulsion a particle in the direction of $\varphi_i$. The right hand side of the second equation consists of a nonlocal particle interaction via alignment and the Wiener process, which plays a role of an external perturbation, respectively. The strength of nonlocal alignment is controlled with a parameter $\sigma\in\mathbb{R}_+$. Particles interact with their nonlocal neighborhood, which consists of all particles within the distance $\varrho\in\mathbb{U}$ and is defined as
\begin{equation*}
	B_\varrho^i=\left\{ j\in\{1,\dots,N\}\backslash i \mid (x_i-x_j)^2+(y_i-y_j)^2\leq\varrho^2 \right\},
\end{equation*}
and $\vert B_{\varrho}^i \vert$ denotes the neighborhood's cardinality. We postulate that particles' alignment mechanism is subject to a phase lag $\alpha\in\mathbb{T}$. If $\alpha=0$, particles align to the directions of their neighbors. Otherwise, particles perform excessive rotation reminiscent of chiral motion. In a noninteracting regime $\sigma=0$, particles exhibit Brownian motion, modeled by a family of independent Wiener processes $W_i$ with a self-diffusion coefficient $D_\varphi\in\mathbb{R}_+$. As a result, particle's dynamics are determined though an interplay of nonlocal alignment and external stochastic forces. As an initial value problem (IVP), we consider this system together with independent and identically distributed initial data $p_i(0)\in\Omega$ for $i=1,\dots,N$.

Such systems of coupled Langevin equations are a common method to formalize the collective dynamics of interacting particles and they have been extensively investigated on the matter of self-organization phenomena \cite{vicsek:prl} and related phase transitions. The most famous phenomenon is the spontaneous rotational symmetry breaking resulting in the emergence of orientational order, either polar or higher order one. Apart from that, recent attention has been devoted to the revealing of spatially inhomogeneous pattern formation. In regard to the aforementioned model \eqref{eq:chimera_sde}, it has been shown previously \cite{kruk2020pre} that it generates a large variety of intriguing spatially inhomogeneous chiral dynamics, the most prominent of which are traveling bands, dense clouds, and vortices, as well as localized chimera states \cite{kruk:aps2018}.

\subsection{Continuum limit}

When the number of particles becomes large, we look for a continuum description of a particle system. We expect that such a description is more efficient compared to the finite size particle system. The dynamical density functional theory provides a number of ways to derive density functions of particle state variables starting from the Langevin dynamics like \eqref{eq:chimera_sde}. Even though it is possible to derive continuum limit equations in terms of joint many-particle density functions that incorporate interactions of any order, one usually restricts oneself to consider a one-particle density function using a mean-field approximation or a molecular chaos assumption \cite{spohn1991springer}. This approximation postulates that particle correlations are negligible allowing a many-particle density function to be decomposed into a product of one-particle density functions. We note that it has been shown \cite{grossmann2020nat_comm} that the symmetry of the emergent order can be different from the symmetry of particle interactions due to correlations and the standard mean-field approach does not work. In our case, particles are modeled as point masses that interact by alignment only. Therefore, we do not expect the emergence of the order higher than the polar one even if the collisions were considered, and assume the mean-field approximation in our derivation of kinetic equations.

We have shown \cite{kruk2020pre} using the framework of the Fokker-Planck equation \cite{archer:jpa} that the continuum description of \eqref{eq:chimera_sde} with macroscopic scaling \cite{kipnis1998scaling} is given by a one-particle probability density function evolving according to the Vlasov-Fokker-Planck \cite{risken,dobrushin,neunzert} equation
\begin{equation}
\label{eq:continuum_chimera_model}
	\begin{split}
	&\partial_t f(r,\varphi,t) = -v_0e(\varphi)\cdot\nabla_r f(r,\varphi,t) - \partial_\varphi \Bigl[ w[f](r,\varphi,t)f(r,\varphi,t) \Bigr] + D_\varphi\partial_{\varphi\varphi} f(r,\varphi,t)
\end{split}
\end{equation}
subject to an initial condition:
\begin{equation}
\label{eq:continuum_chimera_model_initial_condition}
	f(r,\varphi,0) = f_0(r,\varphi),\quad f_0(r,\varphi)\geq0,\quad \int_\Omega f_0(r,\varphi)\; \mathrm{d}r\mathrm{d}\varphi = 1.
\end{equation}
The density function $f(r,\varphi,t)$ quantifies the probability to find a particle at a given position $r\in\mathbb{U}^2$ with a given orientation $\varphi\in\mathbb{T}$ at time $t$. As before, $e(\varphi)=(\cos\varphi,\sin\varphi)\in\mathbb{S}^1$ is a unit vector in the direction of self-propulsion $\varphi$. Note that $\nabla_r=(\partial_x,\partial_y)$ denotes a spatial gradient. The rotational velocity or torque exerted by a nonlocal neighborhood is found to be
\begin{equation}
\label{eq:velocity_component_in_angular_direction}
	w[f](r,\varphi,t) = \frac{\sigma}{\vert C(r;\varrho)\vert} \int_{C(r;\varrho)} f(r',\varphi',t) \sin(\varphi' - \varphi - \alpha)\; \mathrm{d}r'\mathrm{d}\varphi'.
\end{equation}
The region of nonlocal interaction is now defined as a cylinder
\begin{equation*}
	C(r;\varrho) = \left\{ (r',\varphi')\in\mathbb{U}^2\times\mathbb{T} \mid \Vert r-r' \Vert \leq \varrho \right\} \subset \Omega.
\end{equation*}
The size of the nonlocal neighborhood, which is quantified via $|B_\varrho^i|$ in \eqref{eq:chimera_sde}, is now measured as
\begin{equation}
\label{eq:spatially_nonhomogeneous_neighborhood_mass}
	\vert C(r;\varrho)\vert = \int_{C(r;\varrho)} f(r',\varphi',t)\; \mathrm{d}r'\mathrm{d}\varphi'.
\end{equation}
One should keep in mind that in general this neighborhood mass is time-dependent but we usually omit this dependence for notational simplicity. 

Eq.~\ref{eq:velocity_component_in_angular_direction} determines the polarization of a particle flow around a given point (up to a phase shift $\alpha$). This can be formulated by introducing an interaction kernel $K(r,\varphi)=H_\varrho(r)\sin(\varphi+\alpha)$, which is a product of a Heaviside step function $H_\varrho(r) = H(\varrho - \Vert r\Vert)$, which ensures that only a particle flow within the distance $\varrho$ is accounted for, and a shifted alignment function. We can therefore express the rotational velocity functional in a more general form as 
\begin{equation*}
	w[f](r,\varphi,t) = -\frac{\sigma}{\vert C(r;\varrho)\vert} \left[ K(r,\varphi) * f(r,\varphi,t) \right](r,\varphi,t).
\end{equation*}
Whereas this functional form of the angular velocity $w[f](r,\varphi,t)$ does not change the dynamics of the one-particle density function, the convolutional form of the alignment interaction allows for a substantial decrease of temporal complexity of numerical algorithms by means of the discrete Fourier transform \cite{press2002numerical}.

We observe that by writing the diffusion term as $\partial_{\varphi\varphi}f = \partial_\varphi(f\partial_\varphi\log f)$, we can combine the last two terms in \eqref{eq:continuum_chimera_model} in the form of a gradient flow as
\begin{equation}
\label{eq:gradient_flow_form_for_angular_components}
	-\partial_\varphi\{ [w[f](r,\varphi,t) - D_\varphi \partial_\varphi\ln f(r,\varphi,t)] f(r,\varphi,t) \} =: \partial_\varphi [\partial_\varphi\xi[f](r,\varphi,t)f(r,\varphi,t)].
\end{equation}
The new functional $\xi$ denotes the potential function of the flow in the angular direction. However, we cannot extend this form to all the right hand side of \eqref{eq:continuum_chimera_model} unless we design a system to be spatially homogeneous. For the latter case, we use structure preserving numerical strategies developed for gradient flow structures in the construction of a numerical scheme.

One of the transitions we are interested to investigate is the one in terms of the polarization of a particle flow. This is commonly measured using the global polar order parameter defined as
\begin{equation}
\label{eq:global_polar_order_parameter}
	R(t)e^{i\Theta(t)} = \int_\Omega e^{i\varphi}f(r,\varphi,t)\; \mathrm{d}r\mathrm{d}\varphi.
\end{equation}
Here, the absolute value $R$ gives the aforementioned measure while the phase $\Theta$ can be interpreted as a mean direction of the particle flow. If the flow is completely synchronized so that a $\varphi$-marginal of $f$ is a point mass density, the magnitude equals its maximal value $R=1$. If the flow is uniformly distributed so that $f=\text{const}$, the magnitude equals its minimal value $R=0$. For any partially synchronized solution with respect to the angular variable $\varphi$, the order parameter magnitude assumes intermediate values $R\in(0,1)$. Note that in general the right hand side of \eqref{eq:global_polar_order_parameter} must be normalized but since we consider $f$ as a probability density, the normalization term equals one. In the following, when referring to the order parameter, we will often refer to its magnitude $R$ since it provides the main structural information about the particle flow.

The global polar order parameter \eqref{eq:global_polar_order_parameter} provides a global information about the momentum field, which is not enough in a spatially inhomogeneous context. From the continuum PDE \eqref{eq:continuum_chimera_model}, we see that it is worthwhile to consider a nonlocalized version of \eqref{eq:global_polar_order_parameter} as
\begin{equation}
\label{eq:nonlocal_polar_order_parameter}
	R(r,t)e^{i\Theta(r,t)} = \frac{1}{\vert C(r;\varrho)\vert} \int_{C(r;\varrho)} e^{i\varphi}f(r,\varphi,t)\; \mathrm{d}r\mathrm{d}\varphi.
\end{equation}
One can show that in terms of such a nonlocal polar order field, the PDE \eqref{eq:continuum_chimera_model} becomes
\begin{equation*}
	\partial_t f(r,\varphi,t) = -v_0e(\varphi)\cdot\nabla_r f(r,\varphi,t) - \sigma R(r,t) \partial_\varphi \left[ \sin(\Theta(r,t)-\varphi-\alpha) f(r,\varphi,t) \right] + D_\varphi\partial_{\varphi\varphi} f(r,\varphi,t).
\end{equation*}
The presence of the magnitude $R$ in front of the angular flux emphasizes that the rotational rate of change of a particle flow is proportional to the polarization at that point.

\subsection{Spatially homogeneous formulation}

It is straightforward to check that any constant function satisfies \eqref{eq:continuum_chimera_model}. But since we are interested in probability density functions in $\Omega$ as its solutions, we find
\begin{equation}
\label{eq:uniform_density_function}
	f(r,\varphi,t) = \frac{1}{2\pi}.
\end{equation}
In terms of a particle system, this uniform density function corresponds to the chaotic behavior of the system with particles uniformly distributed in $\mathbb{U}^2$ having orientations uniformly distributed in $\mathbb{T}$. One also says that this solution represents a globally disordered or incoherent state.

In order to find solutions except for the trivial one, we note that the model \eqref{eq:continuum_chimera_model} admits a major simplification if we assume that the solutions are spatially homogeneous, i.e., $f(r,\varphi,t)=f(\varphi,t)$. Under that assumption, equation \eqref{eq:continuum_chimera_model} simplifies to a (1+1)-dimensional PDE, which we can also consider as the continuum Kuramoto-Sakaguchi model \cite{sakaguchi:ptp} with diffusion:
\begin{equation}
\label{eq:continuum_chimera_model_spatially_homogeneous}
	\partial_t f(\varphi,t) = -\partial_\varphi \left[ w[f](\varphi,t) f(\varphi,t) \right] + D_\varphi\partial_{\varphi\varphi} f(\varphi,t), \\
\end{equation}
where the nonlocal interaction term \eqref{eq:velocity_component_in_angular_direction} becomes a global one
$
	w[f](\varphi,t) = \sigma \int_\mathbb{T} f(\varphi',t)\sin(\varphi' - \varphi - \alpha)\; \mathrm{d}\varphi'
$
with the neighborhood mass omitted since $\vert C\vert = \int_\mathbb{T} f(\varphi,t)\; \mathrm{d}\varphi = 1$. We shall consider \eqref{eq:continuum_chimera_model_spatially_homogeneous} with an initial condition:
\begin{equation}
\label{eq:continuum_chimera_model_initial_condition_spatially_homogeneous}
	f(\varphi,0) = f_0(\varphi),\quad f_0(\varphi)\geq0,\quad \int_\mathbb{T} f_0(\varphi)\; \mathrm{d}\varphi = 1.
\end{equation}
Note that this equation is of a gradient form $\partial_t f = \partial_\varphi (f \partial_\varphi \xi)$ with the potential $\xi[f](\varphi,t) = -\sigma \int_{\mathbb{T}} \cos(\varphi' - \varphi - \alpha) f(\varphi',t) \;\mathrm{d}\varphi' + D_\varphi \ln f(\varphi,t)$. In the absence of rotations, i.e., $\alpha=0$, the free energy associated to equation \eqref{eq:continuum_chimera_model_spatially_homogeneous} is given by \cite{villani2003ams}
\begin{equation*}
	E[f](t) = -\frac{\sigma}{2} \int_{\mathbb{T}} (\cos * f)(\varphi,t) f(\varphi,t) \;\mathrm{d}\varphi + D_\varphi \int_{\mathbb{T}} f(\varphi,t) \ln f(\varphi,t) \;\mathrm{d}\varphi.
\end{equation*}
We can therefore represent \eqref{eq:continuum_chimera_model_spatially_homogeneous} in a general gradient flow structure $\partial_t f = \partial_\varphi \left(f \partial_\varphi \frac{\delta E[f]}{\delta f}\right)$. Moreover, one can show that this energy functional decays along solutions of \eqref{eq:continuum_chimera_model_spatially_homogeneous} according to
\begin{equation*}
	\frac{\mathrm{d} E}{\mathrm{d} t}[f](t) = - \int_{\mathbb{T}} (\partial_\varphi \xi) ^ 2 f(\varphi,t) \;\mathrm{d}\varphi.
\end{equation*}
However, for chiral interactions with $|\alpha|>0$, the interaction potential is not symmetric and we cannot write down a respective Liapunov functional. Therefore, we will not consider the free energy dissipation of the constructed numerical schemes as in \cite{carrillo2019:jcp}.

In the absence of the phase lag, i.e., $\alpha=0$, the particle model \eqref{eq:chimera_sde} and subsequently the mean-field kinetic equation \eqref{eq:continuum_chimera_model_spatially_homogeneous} can alternatively be formulated in a Cartesian representation, where particle's orientation is modeled as a unit vector in $\mathbb{S}\subset\mathbb{R}^2$ \cite{frouvelle2012siam_jma} and it is updated by projecting the mean contribution of all the neighbors into a subspace orthogonal to particle's orientation in order to keep the velocity constant in magnitude due to the phenomenology of the Vicsek model \cite{vicsek:prl}. In this representation, for sufficiently small diffusion, one observes partial synchronization which is described by a Fisher-von Mises distribution in the continuum limit \cite{frouvelle2012siam_jma}. However, when particle interactions are shifted due to the phase lag, i.e., $\alpha>0$, we cannot perform the same change of variables and obtain a Fisher-von Mises distribution shifted by $\alpha$ along the unit circle. In fact, for sufficiently small diffusion, the partial synchronization is described by a density function which is not symmetric (cf. Fig.~\ref{fig:test_case_nonstationary_phase_synchronization_1d}(a)) in terms of the direction of motion \cite{kruk2020pre}. This is the main novelty of the model under consideration as we show in Proposition \ref{prop:traveling_wave_solution}.

The Kuramoto model is well-known nowadays and it is often used to model synchronization phenomena in various systems \cite{kuramoto1984,pikovsky2003,acebron:rmp,CGPS20}. It was discovered in \cite{kuramoto2002:npcs} that the addition of a phase lag parameter $\alpha$ to the model allows one to obtain a new type of solutions, termed chimera states \cite{abrams:prl}, where both synchronized and disordered populations of oscillators coexist. We showed in \cite{kruk:aps2018} that by extending an oscillator model to a self-propelled particle model, we can obtain such chimera states with and without spatial homogeneity. It appears that for spatially homogeneous states in the continuum limit with noise, we can find a closed form expression for a corresponding density function.

We note that for spatially homogeneous systems nonlocal \eqref{eq:nonlocal_polar_order_parameter} and global \eqref{eq:global_polar_order_parameter} polar order parameters become equal
\begin{equation}
\label{eq:polar_order_parameter_homogeneous}
	R(t)e^{i\Theta(t)} = \int_\mathbb{T} e^{i\varphi}f(\varphi,t)\; \mathrm{d}\varphi
\end{equation}
and one can write the rotational velocity functional \eqref{eq:velocity_component_in_angular_direction} in terms of this polar order parameter as
\begin{equation*}
	w[f](\varphi,t) = \sigma R(t)\sin(\Theta(t)-\varphi-\alpha).
\end{equation*}

\begin{proposition}
\label{prop:traveling_wave_solution}
	The explicit solution of the IVP \eqref{eq:continuum_chimera_model_spatially_homogeneous}-\eqref{eq:continuum_chimera_model_initial_condition_spatially_homogeneous} is given in a traveling wave form $f(\varphi,t) = g(\varphi - vt) = g(\omega)$, where $v\in\mathbb{R}$ is its velocity, with the following profile:
	\begin{equation}
	\label{eq:traveling_wave_solution_profile}
		g(\omega) = c_0 \exp\left[ -\frac{v}{D_\varphi}\omega + \frac{\sigma R}{D_\varphi} \cos(\omega+\alpha) \right] 
		\left( 1 + \left( e^{2\pi\frac{v}{D_\varphi}} - 1 \right) \frac{\int_{0}^{\omega} \exp\left[ \frac{v}{D_\varphi}\omega' - \frac{\sigma R}{D_\varphi} \cos(\omega'+\alpha) \right] \mathrm{d}\omega'}{\int_\mathbb{T} \exp\left[ \frac{v}{D_\varphi}\omega' - \frac{\sigma R}{D_\varphi} \cos(\omega'+\alpha) \right] \mathrm{d}\omega'} \right).
	\end{equation}
    This profile is the solution only if $R$ and $v$ satisfy \eqref{eq:polar_order_parameter_homogeneous}.
\end{proposition}
\begin{proof}
	We note that in the presence of the phase lag, i.e., $\alpha > 0$, the particle flow moves uniformly either to the left or to the right depending on the sign of $\alpha$. Therefore, we are looking for solutions in the form of a traveling wave. We introduce the ansatz $f(\varphi,t) = g(\varphi-vt) = g(\omega)$, where $v$ is the speed of the traveling wave, which is unknown. After this substitution, \eqref{eq:continuum_chimera_model_spatially_homogeneous} becomes
	\begin{equation*}
		D_\varphi \frac{\mathrm{d}^2}{\mathrm{d}\omega^2}g(\omega) + \frac{\mathrm{d}}{\mathrm{d}\omega}\left\{\left[ v + \sigma R\sin(\omega+\alpha) \right] g(\omega)\right\} = 0,
	\end{equation*}
	where due to the fact that \eqref{eq:continuum_chimera_model_spatially_homogeneous} is invariant under phase translations $f(\varphi,t) \mapsto f(\varphi+\varphi_0, t) \quad \forall\varphi_0\in\mathbb{T}$, the order parameter angle can be put to zero in the traveling wave profile without loss of generality. Integrating the above equation with respect to $\omega$ yields
	\begin{equation*}
		D_\varphi \frac{\mathrm{d}}{\mathrm{d}\omega} g(\omega) + \left[ v + \sigma R\sin(\omega+\alpha) \right] g(\omega) = c_1,
	\end{equation*}
	where $c_1\in\mathbb{R}$ is some constant. Solving this equation, we find the solution to be
	\begin{equation*}
		g(\omega) = \exp\left[ -\frac{v}{D_\varphi}\omega + \frac{\sigma R}{D_\varphi} \cos(\omega+\alpha) \right] \left( c_1 \int \exp\left[ \frac{v}{D_\varphi}\omega' - \frac{\sigma R}{D_\varphi} \cos(\omega'+\alpha) \right] \mathrm{d}\omega' + c_2 \right),
	\end{equation*}
	where $c_2\in\mathbb{R}$ is some constant. One of the constants is fixed due to the periodicity constraint, i.e., $g(0) = g(2\pi)$. Namely, this implies
	\begin{equation*}
		c_2 = \exp\left( -\frac{v}{D_\varphi}2\pi \right) \left( c_1 \int_\mathbb{T} \exp\left[ \frac{v}{D_\varphi}\omega' - \frac{\sigma R}{D_\varphi} \cos(\omega'+\alpha) \right] \mathrm{d}\omega' + c_2 \right).
	\end{equation*}
	This subsequently implies
	\begin{equation*}
		c_1 = \frac{c_2 \left( \exp\left( \frac{v}{D_\varphi}2\pi \right) - 1 \right)}{\int_\mathbb{T} \exp\left[ \frac{v}{D_\varphi}\omega' - \frac{\sigma R}{D_\varphi} \cos(\omega'+\alpha) \right] \mathrm{d}\omega'}.
	\end{equation*}
	Next, due to the normalization condition $\int_\mathbb{T} g(\omega)\; \mathrm{d}\omega = 1$, we can find $c_2$ and we put $c_2=c_0$ as a normalization constant. As a result, we find \cite{gupta:jsm} the traveling wave profile \eqref{eq:traveling_wave_solution_profile}. This profile depends on the unknown order parameter magnitude $R$ and the unknown traveling wave velocity $v$, which are not arbitrary but must satisfy the compatibility condition \eqref{eq:polar_order_parameter_homogeneous}.
\end{proof}
The traveling wave profile \eqref{eq:traveling_wave_solution_profile} is a periodic skewed unimodal function with the velocity $v$ being determined by the model parameters. Its skewness coefficient depends in a nonlinear way on the phase lag $\alpha$ and the diffusion $D_\varphi$ parameters \cite{kruk2020pre}. Note that one may apply the ansatz backwards in order to obtain the complete form $f(\varphi,t)$ of the solution of \eqref{eq:continuum_chimera_model_spatially_homogeneous} but for the subsequent numerical analysis, we will use its profile solely. We refer to \cite{kruk2020pre} for other forms of this traveling wave solution.

\begin{remark}
	In the absence of the phase lag, i.e., for $\alpha=0$, the traveling wave solution \eqref{eq:traveling_wave_solution_profile} becomes stationary and it simplifies to
	\begin{equation*}
		f(\varphi) = \frac{\exp\left( \dfrac{\sigma R}{D_\varphi} \cos\varphi \right)}{2\pi I_0\left(\dfrac{\sigma R}{D_\varphi}\right)},
	\end{equation*}
	where $I_0$ denotes the modified Bessel function of the first kind of order zero \cite{olver:nist_handbook}. Moreover, the order parameter magnitude satisfies
	$
		R = I_1(\sigma R / D_\varphi) / I_0(\sigma R / D_\varphi)
	$
	and the order parameter angle $\Theta=0$ without loss of generality due to the translational invariance of \eqref{eq:continuum_chimera_model_spatially_homogeneous}.
	We refer the interested reader to \cite{carrillo2019:jcp} for the detailed analysis of such states.
\end{remark}

Regarding the remark and the emergence of global polar order, we would like to mention an alternative formulation of the Kuramoto model with noise as a self-propelled particle model \cite{grossmann2016pre}. In \cite{grossmann2016pre}, each oscillator is a random walker that carries an internal clock. The model determines the rotational motion of the clock through the interplay of alignment interaction and the Gaussian white noise, whereas the translational particle motion is described by a generic $\alpha$-stable L\'evy noise. Namely, for $\alpha=2$, particles undergo Brownian motion and for $\alpha\in(0,2)$, they perform L\'evy flights. It was demonstrated that such motile oscillators remain disordered in the continuum limit in the former case of diffusive transport but do synchronize in the latter case of superdiffusive transport. In contrast, our model describes particles whose translational motion is exclusively determined by their orientations \eqref{eq:chimera_sde}. Therefore, the ability of particles to synchronize is not affected by the processes described in \cite{grossmann2016pre} and we still observe qualitatively similar phase transitions as in the continuum Kuramoto model with noise \cite{carrillo2019:jcp}.

We next provide the condition for the order parameter magnitude to monotonically decrease when $|\alpha|>0$.
\begin{lemma}
	Let $f$ be a smooth solution to \eqref{eq:continuum_chimera_model_spatially_homogeneous}-\eqref{eq:continuum_chimera_model_initial_condition_spatially_homogeneous}. Then, the order parameter magnitude $R(t)$, defined by \eqref{eq:polar_order_parameter_homogeneous}, satisfies the following equation:
	\begin{equation*}
		\dot{R}(t) = \sigma R(t) \left( \cos\alpha \int_\mathbb{T} \sin^2(\Theta(t)-\varphi) f(\varphi,t)\; \mathrm{d}\varphi - \half\sin\alpha \int_\mathbb{T} \sin(2(\Theta(t)-\varphi)) f(\varphi,t)\; \mathrm{d}\varphi \right) - D_\varphi R(t).
	\end{equation*}
	In particular, if diffusion is higher than a threshold value $D_\varphi \geq D_\varphi^* = \sigma(\cos\alpha+\half\sin|\alpha|)$ for $|\alpha|\leq\frac{\pi}{2}$, then $\dot{R}\leq0$ for all $t\geq0$.
\end{lemma}
\begin{proof}
	From the definition of the spatially homogeneous polar order parameter \eqref{eq:polar_order_parameter_homogeneous}, we can write
	\begin{equation*}
		R(t) = \int_\mathbb{T} \cos(\varphi-\Theta(t))f(\varphi,t)\; \mathrm{d}\varphi,\qquad \int_\mathbb{T} \sin(\varphi-\Theta(t))f(\varphi,t)\; \mathrm{d}\varphi = 0.
	\end{equation*}
	By differentiating the first equation, we find
	\begin{equation*}
	\begin{split}
		\dot{R}(t) &= \dot{\Theta}(t) \int_\mathbb{T} \sin(\varphi-\Theta(t))f(\varphi,t)\; \mathrm{d}\varphi + \int_\mathbb{T} \cos(\varphi-\Theta(t))\partial_t f(\varphi,t)\; \mathrm{d}\varphi \\
		&= \sigma R(t) \int_\mathbb{T} \sin(\Theta(t)-\varphi) \sin(\Theta(t)-\varphi-\alpha)f(\varphi,t)\; \mathrm{d}\varphi - D_\varphi R(t) \\
		&= \sigma R(t) \left( \cos\alpha \int_\mathbb{T} \sin^2(\Theta(t)-\varphi) f(\varphi,t)\; \mathrm{d}\varphi - \half\sin\alpha \int_\mathbb{T} \sin(2(\Theta(t)-\varphi)) f(\varphi,t)\; \mathrm{d}\varphi \right) - D_\varphi R(t).
	\end{split}
	\end{equation*}
\end{proof}

\section{Numerical scheme\label{sec:numerical_scheme}}

\subsection{Phase space discretization}

Our main goal is to develop a finite volume scheme that reproduces a correct behavior of nonlocally interacting particle flow governed by a 3+1 dimensional integro-differential PDE \eqref{eq:continuum_chimera_model}. In this section, we develop a numerical scheme to solve the inhomogeneous kinetic equation \eqref{eq:continuum_chimera_model} compared to \cite{carrillo:commun_comput_phys,carrillo2019:jcp} devoted to the corresponding homogeneous problem. From the numerical analysis viewpoint, we extend to the inhomogeneous kinetic setting by dimensional splitting using the method in \cite{carrillo:commun_comput_phys,carrillo2019:jcp} to deal with the angular variables and Strang splitting to couple with the spatial advection. Our focus in the next section is to utilize this finite volume scheme for the inhomogeneous problem in order to carefully study phase transitions involving skewed distributions and spatially inhomogeneous chimera states.
We quickly review in the next subsection the numerical scheme in \cite{carrillo:commun_comput_phys,carrillo2019:jcp} for the spatially homogeneous system \eqref{eq:continuum_chimera_model_spatially_homogeneous} by introducing a one-dimensional finite volume scheme, i.e., for density functions of the angular variable conveniently adapted to deal with the phase lag. We next proceed to the description of a complete three-dimensional scheme that is applied to density functions of spatial and angular variables. For both proposed methods, we prove mass and positivity preservation as well as derive CFL conditions on their stability. 

The approach we are going to pursue is the following. First, we perform a phase-space discretization of a PDE of interest, thereby deriving a semidiscrete system of ODE for finite volume cells. We derive the set of equations on uniform meshes but the generalization to nonuniform ones is straightforward. Afterwards, by noting that the dynamics of a velocity field in spatial and angular directions qualitatively differs, we attempt a dimensionality splitting technique in order to effectively cope with the dynamics changes due to spatial and angular fluxes. As a result, we obtain a FVM that is second order accurate both in time and in phase-space variables.

\subsubsection{One-dimensional scheme for spatially homogeneous PDEs}

In this section, we develop a finite volume scheme for continuum limit PDEs under the assumption of spatial homogeneity, i.e., for equation of the form \eqref{eq:continuum_chimera_model_spatially_homogeneous}. Let $\mathbb{T}_L=\{0,\dots,L-1\}$ denote a discreet one-dimensional torus with $L$ points. We divide a domain $\mathbb{T}$ into finite volume cells $C_k=[\varphi_{k-\half},\varphi_{k+\half}], k\in\mathbb{T}_L$ of a uniform length $\Delta\varphi=2\pi/L$ with the center of a cell $\varphi_k=k\Delta\varphi$, which correspond to a site $k$ in the torus $\mathbb{T}_L$. Note that since the space is periodic, we have $\varphi_k=\varphi_{k+L},k\in\mathbb{T}_L$.

We define the cell averages \cite{carrillo:commun_comput_phys} $f_k:\mathbb{T}_L\times\mathbb{R}_+ \rightarrow \mathbb{R}$ of a solution to a PDE to be
\begin{equation*}
	f_{k}(t) = \frac{1}{\Delta \varphi}\int_{C_k} f(\varphi,t) \mathrm{d}\varphi.
\end{equation*}
The cell averages $f_k$ are functions of time but for the sake of compactness, we will henceforth omit the explicit time dependence of the computed quantities.

The semidiscrete finite volume scheme is obtained by integrating the PDE \eqref{eq:continuum_chimera_model_spatially_homogeneous} over each cell $C_k,k\in\mathbb{T}_L$. It is consequently formulated as the following system of ODEs for the cell averages:
\begin{equation}
\label{eq:discretized_ode_1d}
	\frac{d}{dt}f_k = - \frac{F_{k+\half}^\varphi - F_{k-\half}^\varphi}{\Delta \varphi},
\end{equation}
where $F_{k\pm\half}^\varphi$ denote angular fluxes. Note that the right hand side of this expression is a second order centered difference of an original flux. In order to find the numerical approximations of the above fluxes at cell interfaces, we need to be able to compute the corresponding values of a solution itself as well as a velocity field. In this paper, we adopt a piecewise linear reconstruction of the numerical solution $f(\varphi,t)$ at each time point. Saying that, we represent a density function in each cell as a first order polynomial as
\begin{equation}
\label{eq:piecewise_linear_reconstruction_1d}
	\tilde{f}(\varphi,t) = f_k(t) + (\partial_\varphi f)_k^{}(\varphi-\varphi_k),\quad \varphi\in C_k,
\end{equation}
where $(\partial_\varphi f)_k^{}$ is a cell average of a partial derivative with respect to $\varphi$, which is to be determined for this reconstruction to work. The knowledge of values of a solution at neighboring cell centers allows us to approximate the slopes $(\partial_\varphi f)_k^{}$ using a second order centered difference method:
\begin{equation*}
	(\partial_\varphi f)_k^{} = \frac{f_{k+1}-f_{k-1}}{2\Delta \varphi}.
\end{equation*}
Unfortunately, it might occur that this slope approximation might lead to negative values of a reconstructed numerical solution \eqref{eq:piecewise_linear_reconstruction_1d}, which we intent to circumvent. For such cases, we recalculate the slope by imposing a slope limiter that keeps reconstructed values nonnegative. In this paper, we chose to use a generalized minmod limiter
\begin{eqnarray*}
	(\partial_\varphi f)_k^{} = \minmod\left( \theta\frac{f_{k+1}-f_k}{\Delta\varphi}, \frac{f_{k+1}-f_{k-1}}{2\Delta\varphi}, \theta\frac{f_k-f_{k-1}}{\Delta\varphi} \right),
\end{eqnarray*}
defined as follows
\begin{equation}
\label{eq:minmod_slope_limiter}
	\minmod(a,b,c)\coloneqq \begin{cases} \min(a,b,c) & a>0,b>0,c>0, \\ \max(a,b,c) & a<0,b<0,c<0, \\ 0 & \text{otherwise}. \end{cases}
\end{equation}
Note that the values, which are corrected with this slope limiter, are generally of first order. However, in all numerical tests we present in this paper, it is practically not imposed and the numerical scheme stays effectively of second order in $\Delta\varphi$.

At this point, the piecewise linear reconstruction is defined and we can apply \eqref{eq:piecewise_linear_reconstruction_1d} in the calculation of numerical fluxes, required in \eqref{eq:discretized_ode_1d}. First, we need to know the values of a solution at each cell interface. They are computed as
\begin{equation}
\label{eq:solution_at_cell_interface_1d}
	f_k^\text{T} = \tilde{f}(\varphi_{k+\half}\!-\!0) = f_k + \frac{\Delta \varphi}{2}(\partial_\varphi f)_k^{},\quad f_k^\text{B} = \tilde{f}(\varphi_{k-\half}\!+\!0) = f_k - \frac{\Delta \varphi}{2}(\partial_\varphi f)_k^{},
\end{equation}
where $f(\varphi_{k+\half}\!-\!0)$ and $f(\varphi_{k-\half}\!+\!0)$ denote function values at cell interfaces $\varphi_{k+\half}$ and $\varphi_{k-\half}$ from inside a cell $C_k$, respectively. We use the cell interface values to define the numerical fluxes in \eqref{eq:discretized_ode_1d} as upwind fluxes as
\begin{eqnarray}
\label{eq:upwind_flux_1d}
	F_{k+\half}^\varphi = w_{k+\half}^+ f_k^\text{T} + w_{k+\half}^- f_{k+1}^\text{B},
\end{eqnarray}
where positive and negative parts of angular velocities are denoted by
\begin{equation}
\label{eq:positive_negative_part_of_velocity}
	w_{k+\half}^+ = \max(w_{k+\half},0),\quad w_{k+\half}^- = \min(w_{k+\half},0).
\end{equation}
The exact velocities themselves are given in the PDE \eqref{eq:continuum_chimera_model_spatially_homogeneous} but need to be numerically approximated at cell interfaces. We note that the PDE is of a gradient flow structure, i.e., we can write $w[f](\varphi,t)=-\partial_\varphi\xi[f](\varphi,t)$ for some potential function $\xi$, as we showed in \eqref{eq:gradient_flow_form_for_angular_components}, and we use this fact to define velocities at cell interfaces using a second order centered difference method as
\begin{equation}
\label{eq:velocity_approximation_at_cell_interface_1d}
	w_{k+\half} = -\frac{\xi_{k+1} - \xi_{k}}{\Delta\varphi}.
\end{equation}
The velocity potential for a homogeneous system is defined as
\begin{equation*}
	\xi[f](\varphi,t) = -\sigma \int_\mathbb{T} f(\varphi',t)\cos(\varphi'-\varphi-\alpha) \mathrm{d}\varphi' + D_\varphi\ln f(\varphi,t),
\end{equation*} 
therefore, its numerical approximation proceeds as follows
\begin{equation*}
\begin{split}
	&\int_\mathbb{T} f(\varphi') \cos(\varphi'-\varphi-\alpha) \mathrm{d}\varphi' = \sum_{n\in\mathbb{T}_L}\int_{C_n} [f_n + (\partial_\varphi f)_n(\varphi'-\varphi_n)] \cos(\varphi'-\varphi_k-\alpha) \mathrm{d}\varphi' \\
	&= \sum_{n\in\mathbb{T}_L} f_n \int_{\varphi_{n-\half}}^{\varphi_{n+\half}} \cos(\varphi'-\varphi_k-\alpha) \mathrm{d}\varphi' + \sum_{n\in\mathbb{T}_L} (\partial_\varphi f)_n^{} \int_{\varphi_{n-\half}}^{\varphi_{n+\half}} (\varphi'-\varphi_n)\cos(\varphi'-\varphi_k-\alpha) \mathrm{d}\varphi' \\
	&= \sum_{n\in\mathbb{T}_L} f_{n} \left(2\sin\frac{\Delta\varphi}{2}\right) \cos(\varphi_n-\varphi_k-\alpha) + \sum_{n\in\mathbb{T}_L} (\partial_\varphi f)_n^{} \left(\Delta\varphi \cos\frac{\Delta\varphi}{2} - 2\sin\frac{\Delta\varphi}{2}\right) \sin(\varphi_n-\varphi_k-\alpha).
\end{split}
\end{equation*}
As a result, we have obtained the following representation of an approximated velocity potential, which is to be used in \eqref{eq:velocity_approximation_at_cell_interface_1d}, as
\begin{equation*}
\begin{split}
	\xi_k = - \frac{\sigma}{\sum_{n\in\mathbb{T}_L} f_n} \sum_{n\in\mathbb{T}_L}&\left[ f_n \left(\frac{\sin\frac{\Delta\varphi}{2}}{\frac{\Delta\varphi}{2}}\right) \cos(\varphi_n - \varphi_k - \alpha) \right.\\
	&\left.+ (\partial_\varphi f)_n^{} \sin(\varphi_n - \varphi_k - \alpha) \left(\cos\frac{\Delta\varphi}{2} - \frac{\sin\frac{\Delta\varphi}{2}}{\frac{\Delta\varphi}{2}}\right) \right] + D_\varphi\ln f_k,
\end{split}
\end{equation*}
where we have used the fact that since $f$ is a probability density function, its piecewise linear reconstruction yields $\int_\mathbb{T} \tilde{f}(\varphi,t) \mathrm{d}\varphi = \sum_{n\in\mathbb{T}_L} f_n\Delta\varphi$. We note that the above approximation of the velocity potential is exact in $\Delta\varphi$ given a piecewise linear reconstruction of a density function.

\begin{theorem}
\label{thm:cfl_1d}
	Consider the IVP \eqref{eq:continuum_chimera_model_spatially_homogeneous}-\eqref{eq:continuum_chimera_model_initial_condition_spatially_homogeneous} with periodic boundaries and the semidiscrete FVM \eqref{eq:discretized_ode_1d} with a positivity-preserving piecewise linear reconstruction \eqref{eq:piecewise_linear_reconstruction_1d}. Assume that the system of ODEs \eqref{eq:discretized_ode_1d} is discretized by the forward Euler method or by a higher-order strong stability preserving (SSP) ODE solver, whose time step can be expressed as a convex combination of several forward Euler steps. Then, computed cell averages remain nonnegative $f_k(t)\geq0\; \forall k\in\mathbb{T}_L\; \forall t>0$, provided that the following CFL condition is satisfied:
	\begin{equation*}
		\Delta t \leq \frac{\Delta \varphi}{2c},
	\end{equation*}
	where $c=\max\limits_{k\in\mathbb{T}_L}\left\{ w_{k+\half}^+,-w_{k+\half}^- \right\}$ and the velocities at cell interfaces are defined in \eqref{eq:velocity_approximation_at_cell_interface_1d}.
\end{theorem}
\begin{proof}
	According to the forward Euler method, we discretize \eqref{eq:discretized_ode_1d} as
	\begin{equation*}
		f_k(t+\Delta t) = f_k(t) - \frac{\Delta t}{\Delta\varphi} \left( F_{k+\half}^\varphi - F_{k-\half}^\varphi \right).
	\end{equation*}
	We note that we can express cell averages of a solution as a linear combination of corresponding values at cell interfaces, defined in \eqref{eq:solution_at_cell_interface_1d}, as
	\begin{equation*}
		f_k = \half \left( f_k^\text{T} + f_k^\text{B} \right).
	\end{equation*}
	Using this fact and expressing numerical fluxes as upwind fluxes introduced in \eqref{eq:upwind_flux_1d}, we have
	\begin{equation*}
		f_k(t+\Delta t) = \half \left( f_k^\text{T} + f_k^\text{B} \right) - \frac{\Delta t}{\Delta\varphi} \left( w_{k+\half}^+ f_k^\text{T} + w_{k+\half}^- f_{k+1}^\text{B} - w_{k-\half}^+ f_{k-1}^\text{T} - w_{k-\half}^- f_k^\text{B} \right).
	\end{equation*}
	Now we group the terms according to cell interface values and find
	\begin{equation*}
		f_k(t+\Delta t) = \left(\half-\frac{\Delta t}{\Delta\varphi}w_{k+\half}^+\right) f_k^\text{T} + \left(\half+\frac{\Delta t}{\Delta\varphi}w_{k-\half}^-\right) f_k^\text{B} + \frac{\Delta t}{\Delta\varphi}w_{k-\half}^+ f_{k-1}^\text{T} - \frac{\Delta t}{\Delta\varphi}w_{k+\half}^- f_{k+1}^\text{B}.
	\end{equation*}
	The last two terms are always nonnegative. To guarantee positivity preservation of $f_k(t+\Delta t)$, we must require that the values in parentheses of the first two terms remain nonnegative. This yields the following conditions:
	\begin{equation*}
		\frac{\Delta t}{\Delta\varphi}w_{k+\half}^+ \leq \half,\quad \frac{\Delta t}{\Delta\varphi}w_{k-\half}^- \geq -\half,
	\end{equation*}
	the combination of which gives the desired CFL condition.
\end{proof}

\begin{remark}
We indicate that it was proved in \cite{carrillo:commun_comput_phys} that the one-dimensional numerical scheme considered so far preserves the conservation of mass and provides the decay of discrete free energy for systems with symmetric interaction potential. It also has a very good property: numerical steady states are characterized by $\xi_k = constant$ for all $k$, as in the continuum setting, due to \eqref{eq:velocity_approximation_at_cell_interface_1d}.
\end{remark}

Numerical studies of spatially homogeneous PDEs \eqref{eq:continuum_chimera_model_spatially_homogeneous} using the finite volume scheme of this section are conducted in Sections \ref{sec:stationary_phase_synchronization_1d}-\ref{sec:nonstationary_phase_synchronization_3d}. In the following, we generalize the scheme for general spatially inhomogeneous three-dimensional PDEs \eqref{eq:continuum_chimera_model}.

\subsubsection{Three-dimensional scheme for spatially inhomogeneous PDEs}

We start from the discretization of a phase space $\Omega$ into finite volume cells. Dimensions corresponding to $x$, $y$, and $\varphi$ variables are divided into $N$, $M$, and $L$ cells, respectively. Linear sizes of cells are $\Delta x=1/N$, $\Delta y=1/M$, and $\Delta\varphi=2\pi/L$. Let $\mathbb{U}_N$, $\mathbb{U}_M$, and $\mathbb{T}_L$ denote discreet one-dimensional tori with $N, M, L\in\mathbb{N}$ points, respectively, i.e., $\mathbb{U}_N = \left\{ 0,\dots,N-1 \right\}$, $\mathbb{U}_M = \left\{ 0,\dots,M-1 \right\}$, $\mathbb{T}_L = \left\{ 0,\dots,L-1 \right\}$. We define a uniform grid consisting of cells $C_{i,j,k}=[x_{i-\half},x_{i+\half}]\times[y_{j-\half},y_{j+\half}]\times[\varphi_{k-\half},\varphi_{k+\half}]$ with cell centers $(x_i,y_j,\varphi_k) = (i\Delta x,j\Delta y,k\Delta\varphi), i\in\mathbb{U}_N, j\in\mathbb{U}_M, k\in\mathbb{T}_L$. Due to the periodic boundaries, we have $x_{N+i}=x_i$, $y_{M+j}=y_j$, and $\varphi_{L+k}=\varphi_k$.

The discretization of $\Omega$ consists of three-dimensional cells which can be enumerated with a three-dimensional torus as
\begin{equation*}
	\Omega_{N,M,L} = \left\{ (i,j,k) \mid i\in\mathbb{U}_N, j\in\mathbb{U}_M, k\in\mathbb{T}_L \right\}.
\end{equation*}
Sites of the torus $(i,j,k)\in\Omega_{N,M,L}$ correspond to the points $(x_i,y_j,\varphi_k)=(i\Delta x,j\Delta y,k\Delta\varphi)\in\Omega$ of the original space. Points of the original space $(x,y,\varphi)\in\Omega$ correspond to the sites $([x/\Delta x],[y/\Delta y],[\varphi/\Delta\varphi])\in\Omega_{N,M,L}$ of the torus, where $[]$ stands for the integer part.

We define cell averages \cite{carrillo:commun_comput_phys} $f_{i,j,k}: \Omega_{N,M,L}\times\mathbb{R}_+ \rightarrow \mathbb{R}$ of solutions to PDEs \eqref{eq:continuum_chimera_model} to be
\begin{equation*}
	f_{i,j,k}(t) = \frac{1}{\Delta x \Delta y \Delta \varphi}\iiint_{C_{i,j,k}} f(x,y,\varphi,t) \mathrm{d}x\mathrm{d}y\mathrm{d}\varphi.
\end{equation*}
As before, for the sake of compactness, we will omit the dependence of most computed quantities on time $t$ henceforth.

The semidiscrete finite volume scheme for a three-dimensional system is obtained by integrating the PDE \eqref{eq:continuum_chimera_model} over each cell $C_{i,j,k},(i,j,k)\in\Omega_{N,M,L}$ and is formulated by the following system of ODEs for $f_{i,j,k}$
\begin{equation}
\label{eq:discretized_ode_3d}
	\frac{d}{dt}f_{i,j,k} = - \frac{F_{i+\half,j,k}^x - F_{i-\half,j,k}^x}{\Delta x} - \frac{F_{i,j+\half,k}^y - F_{i,j-\half,k}^y}{\Delta y} - \frac{F_{i,j,k+\half}^\varphi - F_{i,j,k-\half}^\varphi}{\Delta \varphi},
\end{equation}
for $i\in\mathbb{U}_N$, $j\in\mathbb{U}_M$, and $k\in\mathbb{T}_L$.

In order to define the above fluxes, we extend the same piecewise linear reconstruction method from the previous section. The numerical solution in each cell $C_{i,j,k}$ is thus approximated as a polynomial
\begin{equation}
\label{eq:piecewise_linear_reconstruction_3d}
	\tilde{f}(x,y,\varphi) = f_{i,j,k} + (\partial_x f)_{i,j,k}^{}(x-x_i) + (\partial_y f)_{i,j,k}^{}(y-y_j) + (\partial_\varphi f)_{i,j,k}^{}(\varphi-\varphi_k),\quad (x,y,\varphi)\in C_{i,j,k}.
\end{equation}
To be able to use this representation, we need to find each slope $(\partial_x f)_{i,j,k}^{}$, $(\partial_y f)_{i,j,k}^{}$, and $(\partial_\varphi)_{i,j,k}^{}$. To ensure that the solution is second-order accurate, we define the slopes using the centered difference approximations
\begin{equation*}
	(\partial_x f)_{i,j,k}^{} = \frac{f_{i+1,j,k}-f_{i-1,j,k}}{2\Delta x},\quad 
	(\partial_y f)_{i,j,k}^{} = \frac{f_{i,j+1,k}-f_{i,j-1,k}}{2\Delta y},\quad 
	(\partial_\varphi f)_{i,j,k}^{} = \frac{f_{i,j,k+1}-f_{i,j,k-1}}{2\Delta \varphi}.
\end{equation*}
It might occur that a reconstructed solution becomes negative in some cell $C_{i,j,k}$. In such cases, we recalculate a corresponding slope using a slope limiter. In this paper, we use a generalized minmod limiter \eqref{eq:minmod_slope_limiter}, whose application yields
\begin{equation*}
\begin{split}
	(\partial_x f)_{i,j,k}^{} = \minmod\left( \theta\frac{f_{i+1,j,k}-f_{i,j,k}}{\Delta x}, \frac{f_{i+1,j,k}-f_{i-1,j,k}}{2\Delta x}, \theta\frac{f_{i,j,k}-f_{i-1,j,k}}{\Delta x} \right), \\
	(\partial_y f)_{i,j,k}^{} = \minmod\left( \theta\frac{f_{i,j+1,k}-f_{i,j,k}}{\Delta y}, \frac{f_{i,j+1,k}-f_{i,j-1,k}}{2\Delta y}, \theta\frac{f_{i,j,k}-f_{i,j-1,k}}{\Delta y} \right), \\
	(\partial_\varphi f)_{i,j,k}^{} = \minmod\left( \theta\frac{f_{i,j,k+1}-f_{i,j,k}}{\Delta\varphi}, \frac{f_{i,j,k+1}-f_{i,j,k-1}}{2\Delta\varphi}, \theta\frac{f_{i,j,k}-f_{i,j,k-1}}{\Delta\varphi} \right),
\end{split}
\end{equation*}
We note again that the values, which are corrected with such a limiter, become of first order in a corresponding dimension. But since the number of such points is usually small compared to the total number of grid points, the overall order of the scheme is effectively not reduced. In the numerical tests of this paper, these slope limiters have practically not been triggered.

Now that the piecewise linear reconstruction \eqref{eq:piecewise_linear_reconstruction_3d} is completely determined, we are able to compute solution values at each cell interface the following way:
\begin{equation}
\begin{split}
\label{eq:solution_at_cell_interface_3d}
	f_{i,j,k}^\text{E} = \tilde{f}(x_{i+\half}\!-\!0,y_j,\varphi_k) = f_{i,j,k} + \frac{\Delta x}{2}(\partial_x f)_{i,j,k}^{},\quad f_{i,j,k}^\text{W} = \tilde{f}(x_{i-\half}\!+\!0,y_j,\varphi_k) = f_{i,j,k} - \frac{\Delta x}{2}(\partial_x f)_{i,j,k}^{}, \\
	f_{i,j,k}^\text{N} = \tilde{f}(x_i,y_{j+\half}\!-\!0,\varphi_k) = f_{i,j,k} + \frac{\Delta y}{2}(\partial_y f)_{i,j,k}^{},\quad f_{i,j,k}^\text{S} = \tilde{f}(x_i,y_{j-\half}\!+\!0,\varphi_k) = f_{i,j,k} - \frac{\Delta y}{2}(\partial_y f)_{i,j,k}^{}, \\
	f_{i,j,k}^\text{T} = \tilde{f}(x_i,y_j,\varphi_{k+\half}\!-\!0) = f_{i,j,k} + \frac{\Delta \varphi}{2}(\partial_\varphi f)_{i,j,k}^{},\quad f_{i,j,k}^\text{B} = \tilde{f}(x_i,y_j,\varphi_{k-\half}\!+\!0) = f_{i,j,k} - \frac{\Delta \varphi}{2}(\partial_\varphi f)_{i,j,k}^{},
\end{split}
\end{equation}
where $\tilde{f}(x_{i\pm\half}\pm0,y_j,\varphi_k)$, $\tilde{f}(x_i,y_{j\pm\half}\pm0,\varphi_k)$, and $\tilde{f}(x_i,y_j,\varphi_{k\pm\half}\pm0)$ denote reconstructed solution values at cell interfaces from inside the current cell.

We compute all fluxes in the semidiscrete system of ODEs \eqref{eq:discretized_ode_3d} as upwind fluxes as
\begin{equation}
\begin{split}
\label{eq:upwind_flux_3d}
	F_{i+\half,j,k}^x = u_{i+\half,j,k}^+ f_{i,j,k}^\text{E} + u_{i+\half,j,k}^- f_{i+1,j,k}^\text{W}, \\
	F_{i,j+\half,k}^y = v_{i,j+\half,k}^+ f_{i,j,k}^\text{N} + v_{i,j+\half,k}^- f_{i,j+1,k}^\text{S}, \\
	F_{i,j,k+\half}^\varphi = w_{i,j,k+\half}^+ f_{i,j,k}^\text{T} + w_{i,j,k+\half}^- f_{i,j,k+1}^\text{B},
\end{split}
\end{equation}
where $u_{i+\frac{1}{2},j,k}^\pm$, $v_{i,j+\frac{1}{2},k}^{\pm}$, and $w_{i,j,k+\frac{1}{2}}^{\pm}$ denote positive and negative parts of velocities defined as before according to \eqref{eq:positive_negative_part_of_velocity}. The definition of upwind fluxes required the knowledge of velocities at cell interfaces. For advection in spatial directions, they are directly obtained from the PDE \eqref{eq:continuum_chimera_model} by direct substitution of grid points:
\begin{equation}
\label{eq:cell_interface_velocities_xy}
	u_{i+\frac{1}{2},j,k}=u(x_{i+\frac{1}{2}},y_j,\varphi_k) = \cos\varphi_k,\quad v_{i,j+\frac{1}{2},k}=v(x_i,y_{j+\frac{1}{2}},\varphi_k) = \sin\varphi_k.
\end{equation}
For the angular direction, we again use the gradient flow structure of the angular subflow and consider the last two terms in the PDE \eqref{eq:continuum_chimera_model} as defined by a velocity potential $\xi[f](r,\varphi,t)$ \eqref{eq:gradient_flow_form_for_angular_components}. One can show, that for a spatially nonhomogenous system, it reads
\begin{equation}
\label{eq:angular_velocity_potential_3d}
	\xi[f](r,\varphi,t) = -\sigma \frac{\iiint_{C(r;\varrho)} f(r',\varphi',t)\cos(\varphi'-\varphi-\alpha) \mathrm{d}r'\mathrm{d}\varphi'}{\iiint_{C(r;\varrho)} f(r',\varphi',t) \mathrm{d}r'\mathrm{d}\varphi'} + D_\varphi\ln f(r,\varphi,t).
\end{equation} 
We thus determine the angular velocity at cell interfaces from its potential using a second order centered difference scheme
\begin{equation}
\label{eq:cell_interface_velocities_phi}
	w_{i,j,k+\half} = -\frac{\xi_{i,j,k+1}-\xi_{i,j,k}}{\Delta\varphi}.
\end{equation}

Now, we need to perform the discretization of the potential \eqref{eq:angular_velocity_potential_3d} using the piecewise linear reconstruction of the solution \eqref{eq:piecewise_linear_reconstruction_3d}. First, the numerator for $\xi_{i,j,k}$ is calculated as
\begin{equation*}
\begin{split}
	\iiint_{C(r_{i,j};\varrho)} &\tilde{f}(r',\varphi') \cos(\varphi'-\varphi_k-\alpha) \mathrm{d}r'\mathrm{d}\varphi' =  \\
	&= \sum_{(l,m,n)\in C_{N,M,L}(r_{i,j};\varrho)}\iiint_{C_{l,m,n}} [f_{l,m,n} + (\partial_x f)_{l,m,n}(x'-x_l) + (\partial_y f)_{l,m,n}(y'-y_m) \\
	&\qquad\qquad+ (\partial_\varphi f)_{l,m,n}(\varphi'-\varphi_n)] \cos(\varphi'-\varphi_k-\alpha) \mathrm{d}x'\mathrm{d}y'\mathrm{d}\varphi' \\
	&= \sum_{(l,m,n)\in C_{N,M,L}(r_{i,j};\varrho)} f_{l,m,n}\Delta x\Delta y \int_{\varphi_{n-\frac{1}{2}}}^{\varphi_{n+\frac{1}{2}}} \cos(\varphi'-\varphi_k-\alpha) \mathrm{d}\varphi' \\
	&\qquad\qquad+ \sum_{(l,m,n)\in C_{N,M,L}(r_{i,j};\varrho)} (\partial_\varphi f)_{l,m,n}\Delta x\Delta y \int_{\varphi_{n-\frac{1}{2}}}^{\varphi_{n+\frac{1}{2}}} (\varphi'-\varphi_n)\cos(\varphi'-\varphi_k-\alpha) \mathrm{d}\varphi' \\
	&= \sum_{(l,m,n)\in C_{N,M,L}(r_{i,j};\varrho)} f_{l,m,n} \Delta x\Delta y \left(2\sin\frac{\Delta\varphi}{2}\right)  \cos(\varphi_n-\varphi_k-\alpha) \\
	&\qquad\qquad+ \sum_{(l,m,n)\in C_{N,M,L}(r_{i,j};\varrho)} (\partial_\varphi f)_{l,m,n} \Delta x\Delta y \left(\Delta\varphi \cos\frac{\Delta\varphi}{2} - 2\sin\frac{\Delta\varphi}{2}\right) \sin(\varphi_n-\varphi_k-\alpha),
\end{split}
\end{equation*}
where $r_{i,j}=(x_i,y_j)$ and summations run over $C_{N,M,L}(r_{i,j};\varrho) = \{ (l,m,n)\in\Omega_{N,M,L} \mid (x_i-x_l)^2 +(y_j-y_m)^2 \leq \varrho^2 \}$. Second, the denominator takes the form
\begin{equation*}
\begin{split}
	\iiint_{C(r_{i,j};\varrho)} f(r',\varphi') \mathrm{d}r'\mathrm{d}\varphi' &= \sum_{(l,m,n)\in C_{N,M,L}(r_{i,j};\varrho)} \iiint_{C_{l,m,n}} f(r',\varphi') \mathrm{d}r'\mathrm{d}\varphi' \\
	&= \sum_{(l,m,n)\in C_{N,M,L}(r_{i,j};\varrho)} f_{l,m,n} \Delta x\Delta y\Delta\varphi.
\end{split}
\end{equation*} 
As a result, we find the discretized velocity potential $\xi_{i,j,k}$, which is to be used in \eqref{eq:cell_interface_velocities_phi}, to read
\begin{equation*}
\begin{split}
	\xi_{i,j,k} &= - \frac{\sigma}{\sum_{(l,m,n)\in C_{N,M,L}(r_{i,j};\varrho)} f_{l,m,n}} \sum_{(l,m,n)\in C_{N,M,L}(r_{i,j};\varrho)}\Biggl[ f_{l,m,n} \left( \frac{\sin\frac{\Delta\varphi}{2}}{\frac{\Delta\varphi}{2}} \right) \cos(\varphi_n - \varphi_k - \alpha) \\
	&+ (\partial_\varphi f)_{l,m,n} \sin(\varphi_n - \varphi_k - \alpha) \left( \cos\frac{\Delta\varphi}{2} - \frac{\sin\frac{\Delta\varphi}{2}}{\frac{\Delta\varphi}{2}} \right)\Biggr] + D_\varphi\ln f_{i,j,k},
\end{split}
\end{equation*}
which is exact in $\Delta x$, $\Delta y$, and $\Delta\varphi$. 

\begin{theorem}
	Consider the IVP \eqref{eq:continuum_chimera_model}-\eqref{eq:continuum_chimera_model_initial_condition} with periodic boundaries and the semidiscrete FVM \eqref{eq:discretized_ode_3d} with a positivity-preserving piecewise linear reconstruction \eqref{eq:piecewise_linear_reconstruction_3d}. Assume that the system of ODEs \eqref{eq:discretized_ode_3d} is discretized by the forward Euler method or by a higher-order SSP ODE solver, whose time step can be expressed as a convex combination of several forward Euler steps. Then, computed cell averages remain nonnegative $f_{i,j,k}\geq0\; \forall i\in\mathbb{U}_N,\forall j\in\mathbb{U}_M,\forall k\in\mathbb{T}_L, \forall t>0$, provided that the following CFL condition is satisfied:
	\begin{eqnarray*}
		\Delta t \leq \min\left\{ \frac{\Delta x}{6a}, \frac{\Delta y}{6b}, \frac{\Delta \varphi}{6c} \right\}
	\end{eqnarray*}
	with the coefficients $a=\max\limits_{(i,j,k)\in\Omega_{N,M,L}}\left\{ u_{i+\frac{1}{2},j,k}^+,-u_{i+\frac{1}{2},j,k}^- \right\}$, $b=\max\limits_{(i,j,k)\in\Omega_{N,M,L}}\left\{ v_{i,j+\frac{1}{2},k}^+,-v_{i,j+\frac{1}{2},k}^- \right\}$, $c=\max\limits_{(i,j,k)\in\Omega_{N,M,L}}\left\{ w_{i,j,k+\frac{1}{2}}^+,-w_{i,j,k+\frac{1}{2}}^- \right\}$ and the velocities at cell interfaces are defined in \eqref{eq:cell_interface_velocities_xy},\eqref{eq:cell_interface_velocities_phi}.
	
	Moreover, the mass of the discretized system is conserved, i.e.,
	\begin{equation*}
		\frac{d}{dt} \int_\Omega f(r,\varphi,t)\; \mathrm{d}r\mathrm{d}\varphi = 0.
	\end{equation*}
\end{theorem}
\begin{proof}
	The proof of this theorem follows the same lines as in Theorem~\ref{thm:cfl_1d}. We only comment that in this case, one should express cell averages of a solution $f_{i,j,k}$ as a linear combination of corresponding values at cell interfaces, defined in \eqref{eq:solution_at_cell_interface_3d}, as
	\begin{equation*}
		f_{i,j,k} = \frac{1}{6} \left(f_{i,j,k}^\text{E} + f_{i,j,k}^\text{W} + f_{i,j,k}^\text{N} + f_{i,j,k}^\text{S} + f_{i,j,k}^\text{T} + f_{i,j,k}^\text{B}\right).
	\end{equation*}
	
	The proof of the conservation of mass follows the same reasoning as for the one-dimensional case, which can be found in \cite{carrillo:commun_comput_phys}.
\end{proof}


Numerical studies of spatially inhomogeneous PDEs \eqref{eq:continuum_chimera_model} using the finite volume scheme of this section are conducted in Sections \ref{sec:spatially_nonhomogeneous_solutions}-\ref{sec:phase_transitions_of_spatially_nonhomogeneous_solutions}. Up to now, we have presented the second order discretization of the phase space of the problem keeping the time domain continuous. This way, we have formulated the problems of solving PDEs \eqref{eq:continuum_chimera_model} and \eqref{eq:continuum_chimera_model_spatially_homogeneous} as the problems of solving systems of ODEs \eqref{eq:discretized_ode_3d} and \eqref{eq:discretized_ode_1d}, respectively. In the next section, we consider approaches to perform time discretization so as to keep the FVM of second order in time as well.

\subsection{Dimensionality splitting}

This section considers further discretization approaches for three-dimensional PDEs only. As a result of the phase space discretization from the previous section, we reformulate the IVP for PDEs \eqref{eq:continuum_chimera_model}-\eqref{eq:continuum_chimera_model_initial_condition} as the IVP for the system of ODEs \eqref{eq:discretized_ode_3d}, which we state here for convenience:
\begin{equation}
\label{eq:discretized_ode_3d_repeated}
	\frac{d}{dt}f_{i,j,k}^{}(t) = -\frac{F_{i+\half,j,k}^x - F_{i-\half,j,k}^x}{\Delta x} - \frac{F_{i,j+\half,k}^y - F_{i,j-\half,k}^y}{\Delta y} - \frac{F_{i,j,k+\half}^\varphi - F_{i,j,k-\half}^\varphi}{\Delta \varphi},\quad f_{i,j,k}(0) = (f_0)_{i,j,k}
\end{equation}
for $(i,j,k)\in\Omega_{N,M,L}$ and the fluxes are as defined in the previous section. From the PDE \eqref{eq:continuum_chimera_model} itself and from the derivation of the system of ODEs \eqref{eq:discretized_ode_3d_repeated}, we have seen that the fluxes qualitatively differ for spatial and angular dimensions. Therefore, it might become unreasonable to tackle all of them at once. We will employ this fact in the further construction of our FVM.

The class of methods that allow us to separate dynamics of an autonomous system of ODEs into several subsystems is known as splitting methods \cite{hairer2002}. In this paper, we are interested in such splitting methods that lead to the second order accuracy in time. These splitting methods have been classically used in simulations for kinetic equations in plasma physics, see \cite{kraus2017jpp,filbet2001jcp} for instance.
First, we note the following. We can reenumerate the three-dimensional grid of the problem, which is enumerated with three-dimensional indices from $\Omega_{N,M,L}$, with one-dimensional indices ranging from $0$ to $N M L$. In other words, the three-dimensional set of indices $\Omega_{N,M,L}$ can be bijectively mapped into a one-dimensional set with $NML$ indices. This allows us to express the above model shortly as
\begin{equation*}
	\frac{d}{dt}\bar{f}(t) = F(\bar{f}),\quad \bar{f}(0) = \bar{f}_0,
\end{equation*}
where $\bar{f} = (f_0,\dots,f_{NML-1}) \in \mathbb{R}_+^{NML}$. We note that we can separate the vector field on the right hand side as
\begin{equation}
\label{eq:abstract_ode_with_two_flows}
	\frac{d}{dt}\bar{f}(t) = F^{[1]}(\bar{f}) + F^{[2]}(\bar{f}),
\end{equation}
where $F^{[1]}$ and $F^{[2]}$ are defined by spatial and angular fluxes from \eqref{eq:discretized_ode_3d_repeated}, respectively. Now, according to the splitting procedure, instead of the problem \eqref{eq:abstract_ode_with_two_flows}, we consider two subproblems:
\begin{equation}
\label{eq:two_subproblems}
	\frac{d}{dt}\bar{f} = F^{[1]}(\bar{f}) \quad\text{ and }\quad \frac{d}{dt}\bar{f} = F^{[2]}(\bar{f}).
\end{equation}

Let $\phi_t^{[i]}:\mathbb{R}_+^{NML}\rightarrow\mathbb{R}_+^{NML}, i=1,2$ be dynamical flows generated by vector fields $F^{[i]},i=1,2$ \cite{guckenheimer1990,hirsch2004}, respectively. If we know exact flows with respect to $F^{[1]}$ and $F^{[2]}$, we could attempt to construct a numerical scheme of the form
$
	\Phi_{\Delta t} = \phi_{\Delta t}^{[1]} \circ \phi_{\Delta t}^{[2]},
$
which is known as the Lie-Trotter splitting. However, this formula is only first order accurate. If we also consider the adjoint of the previous method given by
$
	\Phi_{\Delta t}^* = \phi_{\Delta t}^{[2]} \circ \phi_{\Delta t}^{[1]}
$
and build the composition of $\Phi$ and $\Phi^*$ with halved step sizes, we obtain by the theorem of the composition of methods \cite{hairer2002}, a new method
\begin{equation*}
	\Phi_{\Delta t} = \phi_{\Delta t / 2}^{[1]} \circ \phi_{\Delta t}^{[2]} \circ \phi_{\Delta t / 2}^{[1]},
\end{equation*}
which is known as the Strang (Marchuk) splitting. The Strang splitting formula is of order 2. However, for that technique to work, we must be able to integrate two subproblems \eqref{eq:two_subproblems} exactly, which is usually not the case. Therefore, our aim is to obtain a second order method which consists of the composition of approximate direct methods only.

\begin{proposition}
	Consider the composition of numerical methods
	\begin{equation}
	\label{eq:composition_method_half_one_half}
		\Psi_{\Delta t} = \Phi_{\Delta t / 2}^{[2]} \circ \Phi_{\Delta t}^{[1]} \circ \Phi_{\Delta t / 2}^{[2]}.
	\end{equation}
	If the methods $\Phi^{[1]}$ and $\Phi^{[2]}$ are of second order at least, the resulting composition method $\Psi$ is of second order.
\end{proposition}
\begin{proof}
	Let us denote approximate solutions to subproblems \eqref{eq:two_subproblems} using numerical flows as $\bar{f}_1(\Delta t/2) = \Phi_{\Delta t/2}^{[2]}(\bar{f}_1(0))$ if $\bar{f}_1(0) = \bar{f}_0$, $\bar{f}_2(\Delta t) = \Phi_{\Delta t}^{[1]}(\bar{f}_2(0))$ if $\bar{f}_2(0) = \bar{f}_1(\Delta t/2)$, and $\bar{f}_3(\Delta t/2) = \Phi_{\Delta t/2}^{[2]}(\bar{f}_3(0))$ if $\bar{f}_3(0) = \bar{f}_2(\Delta t)$. 
	By performing Taylor expansion of each solution up to second order around $\Delta t=0$ and expressing them in terms of a respective previous solution, we obtain the second order Taylor expansion of the exact flow generated by \eqref{eq:abstract_ode_with_two_flows}.
\end{proof}

Note that one may start by considering the approximate flows $\Phi_{\Delta t}^{[1]}$ and $\Phi_{\Delta t}^{[2]}$ for \eqref{eq:two_subproblems} instead of the exact ones $\phi_{\Delta t}^{[1]}$ and $\phi_{\Delta t}^{[2]}$ from the very beginning and combine them as $\Phi_{\Delta t}=\Phi_{\Delta t}^{[1]}\circ\Phi_{\Delta t}^{[2]}$ to obtain a first order approximation to \eqref{eq:abstract_ode_with_two_flows}. By performing the composition of this method with its adjoint with halved step sizes \cite{hairer2002} and replacing all adjoint flows with the direct ones, one obtains
$
	\Psi_{\Delta t} = \Phi_{\Delta t / 2}^{[1]} \circ \Phi_{\Delta t / 2}^{[2]} \circ \Phi_{\Delta t / 2}^{[2]} \circ \Phi_{\Delta t / 2}^{[1]},
$
which can also be shown to be second order accurate in time, provided each of the submethods is of second order too. However, since \eqref{eq:composition_method_half_one_half} requires fewer time steppers, we choose it for all the subsequently reported results. For such an FVM, we state similar results about prositivity preservation and mass conservation, as in previous sections.

\begin{theorem}
	Consider the IVP \eqref{eq:continuum_chimera_model}-\eqref{eq:continuum_chimera_model_initial_condition} with periodic boundaries and the semidiscrete FVM \eqref{eq:discretized_ode_3d_repeated} with dimensionality splitting \eqref{eq:two_subproblems},\eqref{eq:composition_method_half_one_half} with
	\begin{equation*}
		F^{[1]}_{i,j,k} = - \frac{F_{i,j,k+\half}^\varphi - F_{i,j,k-\half}^\varphi}{\Delta \varphi},
	\end{equation*}
	\begin{equation*}
		F^{[2]}_{i,j,k} = -\frac{F_{i+\half,j,k}^x - F_{i-\half,j,k}^x}{\Delta x} - \frac{F_{i,j+\half,k}^y - F_{i,j-\half,k}^y}{\Delta y},
	\end{equation*}
	where the numerical fluxes $F_{i\pm\half,j,k}^x$, $F_{i,j\pm\half,k}^y$, $F_{i,j,k\pm\half}^\varphi$ are defined as in \eqref{eq:upwind_flux_3d} using a positivity-preserving piecewise linear reconstruction \eqref{eq:piecewise_linear_reconstruction_3d}. Assume that each system of ODEs \eqref{eq:two_subproblems} is discretized by a second order method whose time step can be expressed as a convex combination of several forward Euler steps. Then, the computed cell averages remain nonnegative $f_{i,j,k}\geq0, (i,j,k)\in\Omega_{N,M,L}$ for all $t\geq0$, provided that the following CFL condition is satisfied:
	\begin{eqnarray*}
		\Delta t \leq \min\left\{ \frac{\Delta x}{4a}, \frac{\Delta y}{4b}, \frac{\Delta \varphi}{2c} \right\},
	\end{eqnarray*}
	where $a=\max\limits_{i,j,k}\left\{ u_{i+\frac{1}{2},j,k}^+,-u_{i+\frac{1}{2},j,k}^- \right\}$, $b=\max\limits_{i,j,k}\left\{ v_{i,j+\frac{1}{2},k}^+,-v_{i,j+\frac{1}{2},k}^- \right\}$, and $c=\max\limits_{i,j,k}\left\{ w_{i,j,k+\frac{1}{2}}^+,-w_{i,j,k+\frac{1}{2}}^- \right\}$ and the velocities at cell interfaces are defined in \eqref{eq:cell_interface_velocities_xy},\eqref{eq:cell_interface_velocities_phi}.
	
	Moreover, the mass of the discretized system is conserved, i.e.,
	\begin{equation*}
		\frac{d}{dt} \int_\Omega f(r,\varphi,t)\; \mathrm{d}r\mathrm{d}\varphi = 0.
	\end{equation*}
\end{theorem}
\begin{proof}
	The proof of this theorem follows the same lines as in Theorem~\ref{thm:cfl_1d} but it should be applied to each subsystem \eqref{eq:two_subproblems} separately. By expressing cell averages of a solution $f_{i,j,k}$ of the first subsystem as a linear combination of corresponding values at cell interfaces in angular direction, defined in \eqref{eq:solution_at_cell_interface_3d}, as
	\begin{equation*}
		f_{i,j,k} = \frac{1}{2} \left(f_{i,j,k}^\text{T} + f_{i,j,k}^\text{B}\right)
	\end{equation*}
	and a solution of the second subsytem as a linear combination of corresponding values at cell interfaces in spatial directions as
	\begin{equation*}
		f_{i,j,k} = \frac{1}{4} \left(f_{i,j,k}^\text{E} + f_{i,j,k}^\text{W} + f_{i,j,k}^\text{N} + f_{i,j,k}^\text{S}\right),
	\end{equation*}
	the result follows.
	
	The proof of the conservation of mass follows the same reasoning as for the one-dimensional case, which can be found in \cite{carrillo:commun_comput_phys}.
\end{proof}


We implemented the presented finite volume schemes in C++. To be able to perform numerical analysis of the schemes with meaningful phase space discretizations, we parallelized algorithms using the message passing interface (MPI) standard \cite{gropp2014}. Our implementation can be found under \cite{github_kruk}.

\section{Numerical tests\label{sec:numerical_tests}}

In this section, we demonstrate the performance of the developed numerical scheme. As it has been mentioned, the assumption of spatial homogeneity of solutions allows us to greatly simplify the analysis of the system of interest as well as gain theoretical insight on the behavior of its solutions. Therefore, we first conduct numerical experiments under the assumption of spatial homogeneity in Sections \ref{sec:stationary_phase_synchronization_1d}-\ref{sec:phase_transitions_of_spatially_homogeneous_solutions}. Here, we test how our scheme performs on stationary solutions, i.e., when the phase lag is absent $\alpha=0$, and on new traveling-wave solutions, introduced in Proposition~\ref{prop:traveling_wave_solution}, i.e., when $\alpha>0$. Afterwards, we study phase transitions between disordered motion and partial synchronization described by a traveling wave solution solely. We refer the interested reader to \cite{carrillo2019:jcp} for respective numerical studies of steady states. We quantify these phase transitions in terms of a global polar order parameter \eqref{eq:polar_order_parameter_homogeneous} that measures the degree of polarization in a particle flow. In Section~\ref{sec:spatially_nonhomogeneous_solutions}, we test the numerical scheme in a general setup, where nonstationary spatially inhomogeneous solutions are expected to exist. Moreover, we study phase transitions for spatially inhomogeneous solutions, which are not known analytically. Those phase transitions are quantified in terms of an appropriate order parameter that measures the level of spatial localization on a two-dimensional domain with periodic boundaries. These studies are presented in Section~\ref{sec:phase_transitions_of_spatially_nonhomogeneous_solutions}.

\subsection{Error norms}

In the following, we will examine the accuracy of numerical schemes in both spatially homogeneous and inhomogeneous setups. In the former, analytic solutions are known while in the latter, they are not. Therefore, we introduce different norms for different cases. If we know an exact solution, we will use the following norms to quantify convergence errors \cite{sun:journal_of_computational_physics}:
\begin{equation}
\label{eq:error_norms_exact}
\begin{split}
	e_{L^1}^{} &= \sum_{(i,j,k)\in\Omega_{N,M,L}} \iiint_{C_{i,j,k}} \left|\tilde{f}_h(x,y,\varphi,t) - f(x,y,\varphi,t)\right| \mathrm{d}x\mathrm{d}y\mathrm{d}\varphi, \\
	e_{L^2}^{} &= \left( \sum_{(i,j,k)\in\Omega_{N,M,L}} \iiint_{C_{i,j,k}} \left| \tilde{f}_h(x,y,\varphi,t) - f(x,y,\varphi,t) \right|^2 \mathrm{d}x\mathrm{d}y\mathrm{d}\varphi \right) ^ \frac{1}{2}, \\
	e_{L^\infty}^{} &= \max_{(i,j,k)\in\Omega_{N,M,L}} \left|\tilde{f}_h(x_i,y_j,\varphi_k,t) - f(x_i,y_j,\varphi_k,t)\right|,
\end{split}
\end{equation}
where $\tilde{f}_h$ is a numerical solution with the reconstruction defined by \eqref{eq:piecewise_linear_reconstruction_1d} or \eqref{eq:piecewise_linear_reconstruction_3d} and phase space discretization $h$, $f$ is an exact solution, $C_{i,j,k}=[x_{i-\frac{1}{2}},x_{i+\frac{1}{2}}]\times[y_{j-\frac{1}{2}},y_{j+\frac{1}{2}}]\times[\varphi_{k-\frac{1}{2}},\varphi_{k+\frac{1}{2}}],\; i=0,\dots,N-1,j=0,\dots,M-1,k=0,\dots,L-1$ is a cell on a uniform grid with $NML$ points as defined previously. The integrals in the above expressions are computed using Gauss-Legendre quadrature \cite{press2002numerical}. Note that in case we use quasiuniform initial conditions, we need to align both solutions in a proper way. In the reported results, we shift an exact solution such that its first moment coincides with the first moment of the reconstructed solution. 

In cases where we do not have an exact solution, which is the case when we retrieve spatially inhomogeneous patterns, we first compute a reference solution with the finest discretization $h_1$ and compare the rest of the solutions with cruder discretizations $h_2$ to it using
\begin{equation}
\label{eq:error_norms_benchmark}
\begin{split}
	e_{L^1}^{} &= \sum_{(i,j,k)\in\Omega_{N,M,L}} \iiint_{C_{i,j,k}} \left|\tilde{f}_{h_1}^{}(x,y,\varphi,t) - \tilde{f}_{h_2}^{}(x,y,\varphi,t)\right| \mathrm{d}x\mathrm{d}y\mathrm{d}\varphi, \\
	e_{L^2}^{} &= \left( \sum_{(i,j,k)\in\Omega_{N,M,L}} \iiint_{C_{i,j,k}} \left| \tilde{f}_{h_1}^{}(x,y,\varphi,t) - \tilde{f}_{h_2}^{}(x,y,\varphi,t) \right|^2 \mathrm{d}x\mathrm{d}y\mathrm{d}\varphi \right) ^ \frac{1}{2}, \\
	e_{L^\infty}^{} &= \max_{(i,j,k)\in\Omega_{N,M,L}} \left|\tilde{f}_{h_1}^{}(x_i,y_j,\varphi_k,t) - \tilde{f}_{h_2}^{}(x_i,y_j,\varphi_k,t)\right|,
\end{split}
\end{equation}
where the grid size is chosen as a least common multiple in each dimension, i.e. $C_{i,j,k}, i=0,\dots,\lcm(N_1,N_2)-1, j=0,\dots,\lcm(M_1,M_2)-1, k=0,\dots,\lcm(L_1,L_2)-1$. If one needs to compute an error for one and two dimensional domains, it is done straightforwardly by omitting two or one dimensions, respectively, in the above definitions.

\subsection{Stationary phase synchronization (1D)\label{sec:stationary_phase_synchronization_1d}}

We start the inspection of performance of constructed numerical schemes by first analyzing the simpler spatially homogeneous systems, whose time evolution is governed by PDEs \eqref{eq:continuum_chimera_model_spatially_homogeneous}. It is known that for sufficiently high diffusion levels $D_\varphi$ (or equivalently for small coupling coefficients $\sigma$), the asymptotic solution consists of chaotically moving particles, whose distribution is given by a uniform density function \eqref{eq:uniform_density_function}. For sufficiently low diffusion levels (or large coupling coefficients), particles self-organize into spatially homogeneous polar groups, which are stationary in the absence of a phase lag, i.e., for $\alpha=0$, or rotate with constant frequency for $\alpha\neq0$. In this section, we investigate how the numerical scheme performs in the former case. Namely, we consider a spatially homogeneous version of the continuum limit equation with $\alpha=0$, also known as the continuum Kuramoto model \cite{carrillo2019:jcp} with diffusion:
\begin{equation}
\label{eq:spatially_homogeneous_continuum_pde_with_zero_lag}
\begin{cases}
	\partial_t f(\varphi,t) = -\partial_\varphi\left[ w[f](\varphi,t)f(\varphi,t) \right] + D_\varphi\partial_{\varphi\varphi}f(\varphi,t) & \text{ in } \mathbb{T}\times(0,\infty) \\
	f(\varphi,0) = a_0 + \sum\limits_{k=1}^{K} \left[ a_k\cos(k\varphi) + b_k\sin(k\varphi) \right] & \text{ on } \mathbb{T}\times\left\{t=0\right\},
\end{cases}
\end{equation}
where the angular velocity induced by particles' interactions is $w[f](\varphi,t)=\sigma \int_\mathbb{T} f(\varphi',t)\sin(\varphi' - \varphi) \mathrm{d}\varphi'$. The initial condition $f(\varphi,0)$ in the form of a trigonometric series is used to model an irregular but sufficiently smooth function, which is required by the numerical scheme. We shall refer to such initial conditions as quasirandom initial conditions in subsequent discussions. The series coefficients are chosen in such a way that $f(\varphi,0)$ is nonnegative and properly normalized, i.e. $a_0=\frac{1}{2\pi}$, $a_k,b_k\sim\mathcal{U}(-\varepsilon,\varepsilon),k=1,\dots,K$, $K\in\mathbb{N}$. Note that we cannot use a uniform probability density function \eqref{eq:uniform_density_function} as an initial condition since it is already a solution to the problem \eqref{eq:spatially_homogeneous_continuum_pde_with_zero_lag}. We also remark that the normalization of the density function is not generally required by the scheme but continuum limit PDEs, we consider in this paper, describe the behavior of probability density functions.

\begin{figure}[]
	\centering
	\includegraphics[width=1.0\textwidth]{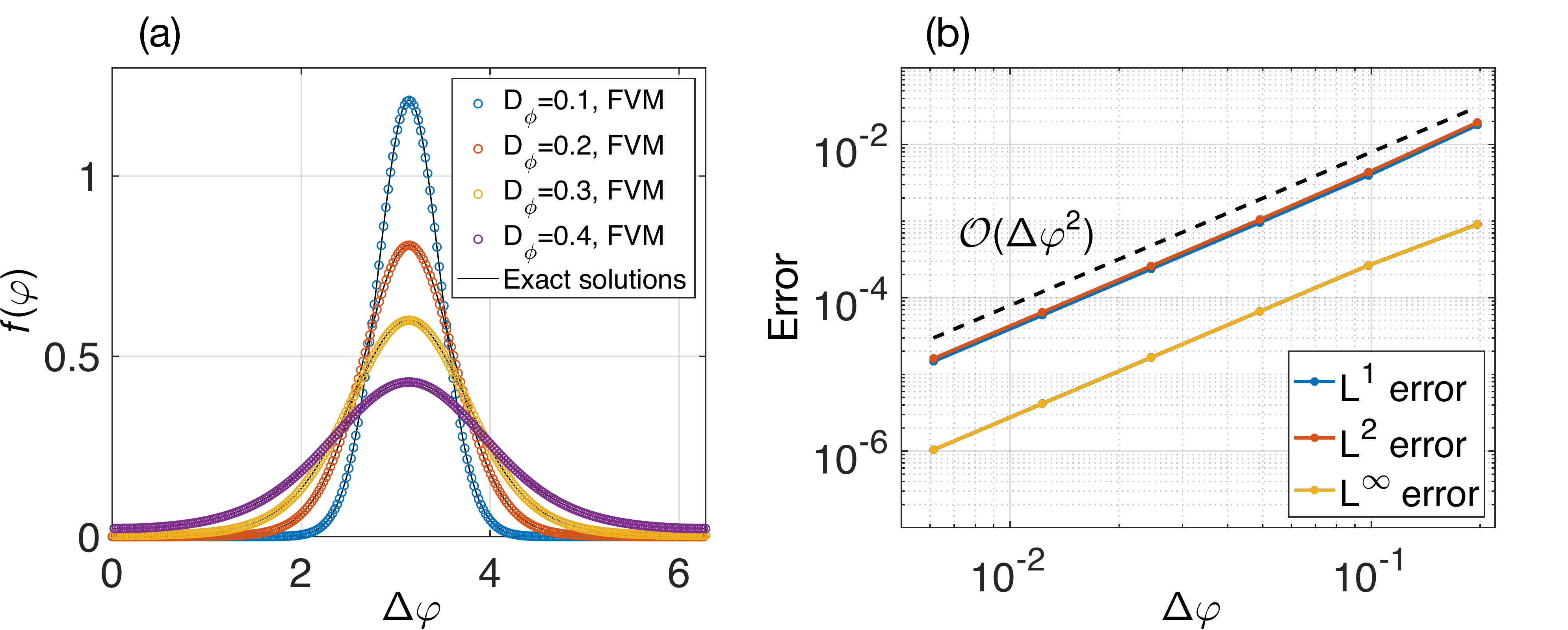}
	\caption
	{
		\label{fig:test_case_stationary_phase_synchronization_1d}
		(a) Numerical stationary solutions of (1+1)-dimensional Eq.~\eqref{eq:spatially_homogeneous_continuum_pde_with_zero_lag} with $L=256$ grid points, compared to respective exact solutions (black lines) in the form of the von Mises density function \eqref{eq:von_mises_density_function}. All solutions are centered such that their means coincide. (b) The convergence of error for the same problem \eqref{eq:spatially_homogeneous_continuum_pde_with_zero_lag} in $L^1$, $L^2$, and $L^\infty$ norms at $t=100$. The points correspond to grids with 32, 64, 128, 256, 512, and 1024 points. Model parameters are $\sigma=1$, $D_\varphi=0.1$.
	}
\end{figure}

It is well known that the problem \eqref{eq:spatially_homogeneous_continuum_pde_with_zero_lag}, i.e. the continuum Kuramoto model for identical oscillators with diffusion, exhibits a second-order phase transition with respect to either coupling strength $\sigma$ or diffusion level $D_\varphi$. The phase transition if of second order and occurs at $D_\varphi = \frac{\sigma}{2}$. Its numerical investigation was already described in detail in \cite{carrillo2019:jcp}, therefore, we do not consider it here. Instead, we only test how our finite volume scheme performs on the solutions of \eqref{eq:spatially_homogeneous_continuum_pde_with_zero_lag} for parameters from the region of stability of a polar order solution. In this case, this solution is a von Mises density function
\begin{equation}
\label{eq:von_mises_density_function}
	f(\varphi)=\frac{\exp\left[ \frac{\sigma R}{D_\varphi} \cos(\varphi-\Theta) \right]}{2\pi I_0\left(\gamma\right)},
\end{equation}
where $\Theta\in\mathbb{T}$ is the average direction of a particle flow, whose value depends on initial conditions, and $I_0$ is the modified Bessel function of the first kind. Fig.~\ref{fig:test_case_stationary_phase_synchronization_1d}(a) illustrates that a numerical solution approximates the exact one very well. Note that since the numerical solution was obtained from quasiuniform initial conditions according to \eqref{eq:spatially_homogeneous_continuum_pde_with_zero_lag}, it is manually centered so that $\Theta=\pi$, for comparison reasons. The numerical solutions are taken at $t=200$, when all of them has converged to steady states. However, such a long time is not required for all presented solutions. The time the system takes to converge to a steady state depends on the value of a diffusion coefficient $D_\varphi$ (or reversely the coupling strength $\sigma$). The closer the value to the order-disorder transition point $D_\varphi=\frac{\sigma}{2}$, the longer the time is. This is a well known bottleneck effect near bifurcation points. For the solution with $D_\varphi=0.4\sigma$, it takes around $t=200$ simulation time units to converge.

The results of error convergence are presented in Fig.~\ref{fig:test_case_stationary_phase_synchronization_1d}(b). One can see that in the current setup, the scheme is second order accurate as is guaranteed by its construction. Numerical solutions were compared to the aforementioned von Mises density function \eqref{eq:von_mises_density_function} using error norms defined by \eqref{eq:error_norms_exact}. Initial conditions were again quasiuniform, the time step was $\Delta t=10^{-5}$, and the errors were computed at $t=100$, when the steady state has been reached.


\subsection{Nonstationary phase synchronization (1D)}

\begin{figure}[]
	\centering
	\includegraphics[width=1.0\textwidth]{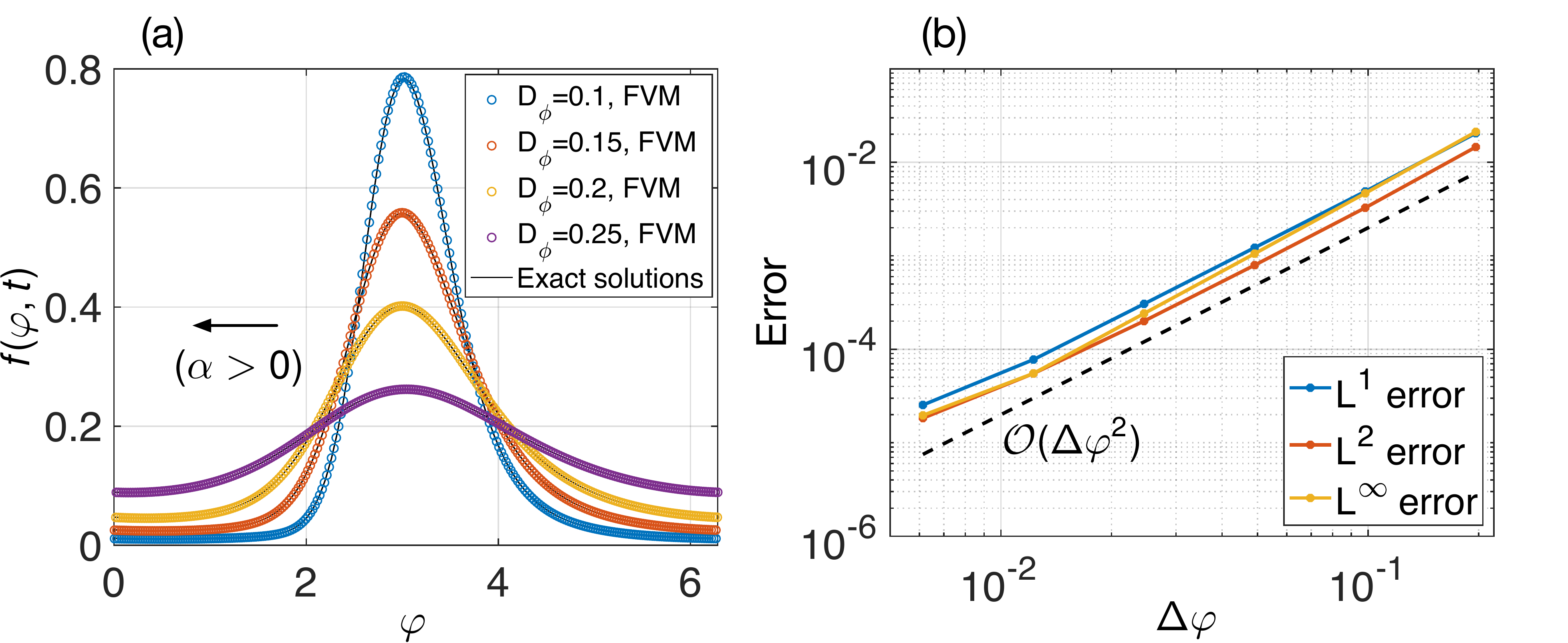}
	\caption
	{
		\label{fig:test_case_nonstationary_phase_synchronization_1d}
		(a) Profiles of numerical traveling wave solutions of (1+1)-dimensional Eq.~\eqref{eq:spatially_homogeneous_continuum_pde_with_nonzero_lag} with $L=256$ grid points, compared to the respective exact solutions in the form of the skewed circular one-peaked density function \eqref{eq:traveling_wave_solution}. All solutions are centered such that their means coincide. The black arrow indicates the direction of motion of traveling waves. (b) The convergence of error for the problem \eqref{eq:spatially_homogeneous_continuum_pde_with_nonzero_lag} in $L^1$, $L^2$, and $L^\infty$ norms at $t=100$. The points correspond to grids with 32, 64, 128, 256, 512, and 1024 points. Other parameters are $\sigma=1$, $\alpha=1$, $D_\varphi=0.1$.
	}
\end{figure}

Next, we keep the assumption of spatial homogeneity but assume $\alpha\neq0$. If the phase lag is added to the particle alignment interaction, a particle flow starts to rotate. Its continuum description in terms of a probability density function becomes skewed and assumes a traveling wave form. A slight generalization to the previous example leads to the following continuum Kuramoto-Sakaguchi model \cite{sakaguchi:ptp} with diffusion:
\begin{equation}
\label{eq:spatially_homogeneous_continuum_pde_with_nonzero_lag}
	\begin{cases}
		\partial_t f(\varphi,t) = -\partial_\varphi\left[ w[f](\varphi,t)f(\varphi,t) \right] + D_\varphi\partial_{\varphi\varphi}f(\varphi,t) & \text{ in } \mathbb{T}\times(0,\infty) \\
		f(\varphi,0) = a_0 + \sum\limits_{k=1}^{K} \left[ a_k\cos(k\varphi) + b_k\sin(k\varphi) \right] & \text{ on } \mathbb{T}\times\left\{t=0\right\},
	\end{cases}
\end{equation}
where $w[f](\varphi,t) = \sigma \int_\mathbb{T} f(\varphi',t)\sin(\varphi' - \varphi - \alpha) \mathrm{d}\varphi'$ and the choice of an initial condition follows the same considerations as in the previous example. This problem has again two solutions, a uniform density function \eqref{eq:uniform_density_function} and a skewed unimodal density function \eqref{eq:traveling_wave_solution_profile}
\begin{equation}
\label{eq:traveling_wave_solution}
\begin{split}
	f(\varphi,t) &= c_0 \exp\left[ -\frac{v}{D_\varphi}\varphi + \frac{\sigma R}{D_\varphi}\cos(\varphi-vt+\alpha) \right] \times \\
	\times &\left( 1 + \left( e^{2\pi \frac{v}{D_\varphi}} - 1 \right) \frac{\int_{0}^{\varphi-vt} \exp\left[ \frac{v}{D_\varphi}\varphi' - \frac{\sigma R}{D_\varphi}\cos(\varphi'+\alpha) \right] \mathrm{d}\varphi'}{\int_\mathbb{T} \exp\left[ \frac{v}{D_\varphi}\varphi' - \frac{\sigma R}{D_\varphi}\cos(\varphi'+\alpha) \right] \mathrm{d}\varphi'} \right),\quad \varphi\in[vt,2\pi+vt)
\end{split}
\end{equation}
where $c_0\in\mathbb{R}$ is a normalization constant. The stability of solutions to the problem \eqref{eq:spatially_homogeneous_continuum_pde_with_nonzero_lag} now depends on the values of a phase lag parameter $\alpha$, diffusion level $D_\varphi$, and coupling strength $\sigma$. With respect to these parameters, the system exhibits a second order phase transition, which we will investigate later (see Section~\ref{sec:phase_transitions_of_spatially_homogeneous_solutions}). The phase transition occurs at $D_\varphi = \frac{\sigma}{2}\cos\alpha$. Here, we illustrate the performance of the scheme for parameter values from the region of stability of the skewed density function \eqref{eq:traveling_wave_solution}, i.e. for $D_\varphi < \frac{\sigma}{2}\cos\alpha$. Fig.~\ref{fig:test_case_nonstationary_phase_synchronization_1d}(a) illustrates that the numerical solution approximates the exact one well. Note that Fig.~\ref{fig:test_case_nonstationary_phase_synchronization_1d}(a) shows profiles of actual solutions which have been manually centered for comparison reasons. The numerical solutions were taken at $t=1000$, which is the time the solution with $\alpha=1$, $D_\varphi=0.25$, $\sigma=1$ takes to converge to the traveling wave form \eqref{eq:traveling_wave_solution}. This is again because of the critical slowing-down close to the phase transition line (see the discussion in the previous section). In contrast to the zero phase lag case, the phase transition occurs for smaller diffusion levels for fixed $\alpha$. Moreover, the time the system takes to converge increases exponentially with $\alpha$ even when diffusion is absent \cite{kruk:aps2018}.

We illustrate the error convergence for parameters from the same stability region (cf. Fig.~\ref{fig:test_case_nonstationary_phase_synchronization_1d}(b)). One can see that the errors are again of second order as expected. The norms were computed according to Eqs.~\eqref{eq:error_norms_exact} with Eq.~\eqref{eq:traveling_wave_solution} as a reference solution. The time step was chosen $\Delta t=10^{-5}$. The norms were computed at $t=100$, when the system had converged to a traveling wave solution.

\subsection{Stationary phase synchronization (3D)}

\begin{figure}[]
	\centering
	\includegraphics[width=0.85\textwidth]{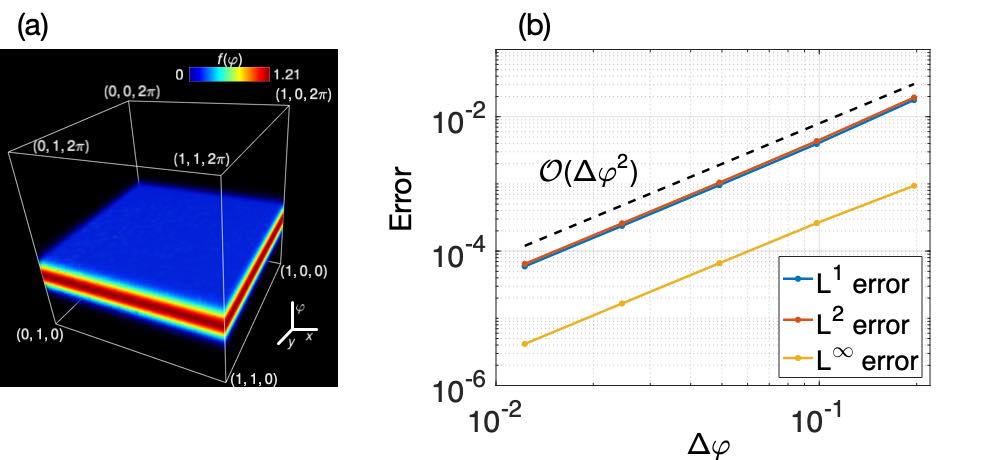}
	\caption
	{
		\label{fig:test_case_stationary_phase_synchronization_3d}
		(a) Numerical stationary solution of (3+1)-dimensional Eq.~\eqref{eq:spatially_nonhomogeneous_continuum_pde_with_zero_lag} with $N\times M\times L=40\times40\times256$ grid points, whose projection onto $\varphi$-axis has the same shape as in Fig.~\ref{fig:test_case_stationary_phase_synchronization_1d}(a). (b) The convergence of error for the problem \eqref{eq:spatially_nonhomogeneous_continuum_pde_with_zero_lag} in $L^1$, $L^2$, and $L^\infty$ norms at $t=100$. The points correspond to grids with 32, 64, 128, 256, and 512 points in $\varphi$. The grid contains $40\times40$ points in spatial variables. Other parameters are $\sigma=1$, $\varrho=0.05$, $D_\varphi=0.1$.
	}
\end{figure}

Spatially homogeneous PDE \eqref{eq:continuum_chimera_model_spatially_homogeneous} was obtained from the original one \eqref{eq:continuum_chimera_model} under the assumption of spatial homogeneity of the system. On one hand, it allowed us to derive analytical results to understand model's behavior but on the other hand, this does not correspond to the original particle flow. We next consider the complete (3+1)-dimensional equation \eqref{eq:continuum_chimera_model} but start from the analysis in a parameter region where spatially homogeneous solutions are stable. Again, we first assume a simpler case where the phase lag is not taken into account, i.e. $\alpha=0$ and consider the following IVP
\begin{equation}
\label{eq:spatially_nonhomogeneous_continuum_pde_with_zero_lag}
\begin{cases}
	\partial_t f(r,\varphi,t) = -\nabla_r\cdot[v_0e(\varphi) f(r,\varphi,t)] -\partial_\varphi\left[ w[f](r,\varphi,t) f(r,\varphi,t) \right] + D_\varphi\partial_{\varphi\varphi}f(r,\varphi,t) & \text{in } \Omega\times(0,\infty) \\
	f(r,\varphi,0) = c_0 +\! \displaystyle\sum\limits_{n,m,l=1}^{K} c_{nml} \sin(2\pi n x - \alpha_{nml}) \sin(2\pi m y - \beta_{nml}) \sin(l \varphi - \gamma_{nml}) & \text{on } \Omega\times\left\{t=0\right\},
\end{cases}
\end{equation}
where $e(\varphi)=(\cos\varphi,\sin\varphi)\in\mathbb{S}^1$ is a unit velocity vector and the angular torque is defined as $w[f](r,\varphi,t) = \dfrac{\sigma}{\vert C(r)\vert} \iiint_{C(r)} f(r',\varphi',t)\sin(\varphi' - \varphi) \mathrm{d}r'\mathrm{d}\varphi'$. Coefficients $c_0,c_{nml}\in\mathbb{R}$ in the initial condition are chosen such that the density function $f(r,\varphi,0)$ is nonnegative and normalized, and the normalization term $\vert C(r)\vert$ is defined by Eq.~\eqref{eq:spatially_nonhomogeneous_neighborhood_mass}. The shifts of the arguments are chosen at random $\alpha_{nml},\beta_{nml},\gamma_{nml} \sim \mathcal{U}(0,2\pi)$ in order to approximate a fluctuating density field. One can consider such an initial condition as a generalization of initial conditions from spatially homogeneous examples \eqref{eq:spatially_homogeneous_continuum_pde_with_zero_lag},\eqref{eq:spatially_homogeneous_continuum_pde_with_nonzero_lag} to the three-dimensional case.

From the linear stability analysis \cite{kruk2020pre}, we know that for $\alpha=0$ the problem \eqref{eq:spatially_nonhomogeneous_continuum_pde_with_zero_lag} has the same set of solutions as the spatially homogeneous problem \eqref{eq:spatially_homogeneous_continuum_pde_with_zero_lag}, i.e. a uniform density function \eqref{eq:uniform_density_function} and the von Mises density function \eqref{eq:von_mises_density_function}. It appears that these solutions are stable against spatially inhomogeneous perturbations in the absence of phase lag. The reason why spatial patterns, such as traveling bands, do not develop for parameter values close to the order-disorder transition line is because the alignment interaction is normalized in $w[f](r,\varphi,t)$. We remark that in case the normalization term $\vert C\vert$ is removed, one actually observes the emergence of such spatial patterns, as is well known for active matter systems with polar interactions \cite{levis2019:prr,nagai2015:prl}. An exemplary solution to the problem \eqref{eq:spatially_nonhomogeneous_continuum_pde_with_zero_lag} is presented in Fig.~\ref{fig:test_case_stationary_phase_synchronization_3d}(a). One can see that the solution is indeed homogeneous with respect to $x$ and $y$ but has a unimodal symmetric profile in $\varphi$. Its projection onto the $\varphi$-axis is qualitatively similar to the one in Fig.~\ref{fig:test_case_stationary_phase_synchronization_1d}(a).

The analysis of solutions for different grid sizes shows the second order convergence in terms of $\Delta\varphi$ (cf. Fig.~\ref{fig:test_case_stationary_phase_synchronization_3d}(b)). Since the steady state is eventually spatially homogeneous, we cannot test the error convergence in terms of $\Delta x$ and $\Delta y$. The errors were computed using Eqs.~\eqref{eq:error_norms_exact} with Eq.~\eqref{eq:von_mises_density_function} as a spatially homogeneous reference solution. The time step was chosen $\Delta t = 5\cdot10^{-4}$. The errors were computed at $t=100$, when the numerical solution had converged to the steady state. The model parameters were taken from the region of stability of partially synchronized motion.

\subsection{Nonstationary phase synchronization (3D)\label{sec:nonstationary_phase_synchronization_3d}}

\begin{figure}[]
	\centering
	\includegraphics[width=0.85\textwidth]{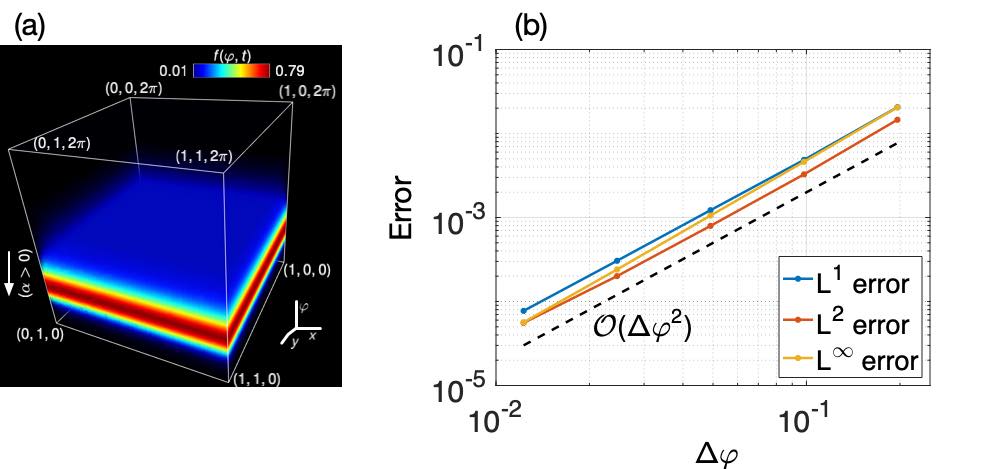}
	\caption
	{
		\label{fig:test_case_nonstationary_phase_synchronization_3d}
		(a) Numerical plane wave solution of (3+1)-dimensional Eq.~\eqref{eq:spatially_nonhomogeneous_continuum_pde_with_nonzero_lag} with $N\times M\times L=40\times40\times256$ grid points (see \cite{bcs_youtube_channel,figshare} for a video of its motion), whose projection onto $\varphi$-axis has the same shape as in Fig.~\ref{fig:test_case_nonstationary_phase_synchronization_1d}(a). The white arrow indicates the direction of motion of the plane wave. (b) The convergence of error for the problem \eqref{eq:spatially_nonhomogeneous_continuum_pde_with_nonzero_lag} in $L^1$, $L^2$, and $L^\infty$ norms at $t=200$. The points correspond to grids with 32, 64, 128, 256, and 512 points in $\varphi$. The grid has $40\times40$ points in spatial variables. Other parameters are $\sigma=1$, $\varrho=0.05$, $\alpha=1$, $D_\varphi=0.1$.
	}
\end{figure}

We now proceed by allowing nonzero phase lag $\alpha\neq0$ in a general three-dimensional system \eqref{eq:continuum_chimera_model}. We know that in addition to a uniform disordered \eqref{eq:uniform_density_function} and spatially homogeneous ordered motion \eqref{eq:traveling_wave_solution}, spatially inhomogeneous solutions emerge for sufficiently high values of $\alpha$. Before we proceed to the analysis of such solutions, we would look into performance of our numerical scheme in a parameter region, where a spatially homogeneous skewed unimodal density function \eqref{eq:traveling_wave_solution} is stable against spatially inhomogeneous perturbations in order to be consistent with the previous development. We consider the following IVP
\begin{equation}
\label{eq:spatially_nonhomogeneous_continuum_pde_with_nonzero_lag}
\begin{cases}
	\partial_t f(r,\varphi,t) = -\nabla_r\cdot[v_0e(\varphi) f(r,\varphi,t)] - \partial_\varphi\left[ w[f](r,\varphi,t) f(r,\varphi,t) \right] + D_\varphi\partial_{\varphi\varphi}f(r,\varphi,t) & \text{in } \Omega\times(0,\infty) \\
	f(r,\varphi,0) = c_0 +\! \displaystyle\sum\limits_{n,m,l=1}^{K} c_{nml} \sin(2\pi n x - \alpha_{nml}) \sin(2\pi m y - \beta_{nml}) \sin(l \varphi - \gamma_{nml}) & \text{on } \Omega\times\left\{t=0\right\},
\end{cases}
\end{equation}
with $w[f](r,\varphi,t) = \dfrac{\sigma}{\vert C(r)\vert} \iiint_{C(r)} f(r',\varphi',t)\sin(\varphi' - \varphi - \alpha) \mathrm{d}r'\mathrm{d}\varphi'$. For $\alpha$ sufficiently small so as to guarantee the stability of spatially homogeneous densities \eqref{eq:traveling_wave_solution}, an exemplary numerical solution is illustrated in Fig.~\ref{fig:test_case_nonstationary_phase_synchronization_3d}(a). One can see that it is indeed homogeneous in $x$ and $y$ but has a characteristic skewed shape in $\varphi$ (the blue region is more pronounced above the plane than below). Its projection onto the $\varphi$-axis gives a qualitatively similar solitary wave as the one in Fig.~\ref{fig:test_case_nonstationary_phase_synchronization_1d}(a). The wave moves transversal to the $\varphi$ axis with some constant velocity $v$, the sign of which depends conversely on $\alpha$. Thus, it has the form of a plane wave in $\Omega$. Such a form of the solution corresponds to a partially synchronized steadily rotating particle flow, which was investigated and termed a nonlocalized self-propelled chimera state in \cite{kruk:aps2018}.

As in the previous example, because of spatial homogeneity of the solution, we are able to determine the order of error convergence versus $\Delta\varphi$ only (cf. Fig.~\ref{fig:test_case_nonstationary_phase_synchronization_3d}(b)). The errors were computed with respect to a spatially homogeneous plane wave solution whose $\varphi$-profile is given by Eq.~\eqref{eq:traveling_wave_solution} using Eqs.~\eqref{eq:error_norms_exact}. The time step was chosen $\Delta t=5\cdot10^{-4}$. The errors were computed at $t=200$, when numerical solutions converged to the plane wave form.

\subsection{Phase transitions of spatially homogeneous solutions\label{sec:phase_transitions_of_spatially_homogeneous_solutions}}

In this section, we analyze how phase transitions of spatially homogeneous solutions are captured by the finite volume scheme. Moreover, we concentrate on the nonstationary case with nonzero phase lag and refer the reader to \cite{carrillo2019:jcp} for the numerical studies of phase transitions of steady states in the continuum Kuramoto model. Since we know that the continuum Kuramoto-Sakaguchi equation \eqref{eq:continuum_chimera_model_spatially_homogeneous} has a solution of the traveling wave form (see Proposition~\ref{prop:traveling_wave_solution}), it is enough to study its profile $g(\omega) = f(\varphi-vt,0) = f(\varphi,t)$, where $v$ is unknown, in the analysis of related phase transitions. 

The phase transitions between spatially homogeneous polar order and disordered motion are commonly quantified with respect to the polar order parameter defined in \eqref{eq:polar_order_parameter_homogeneous}, i.e., it is an average orientation on the unit circle with respect to a given density function. In the traveling wave profile, the polar order parameter can be re-expressed as
\begin{equation*}
	Re^{i\tilde{\Theta}} = \int_{\mathbb{T}} e^{i\omega} g(\omega)\; \mathrm{d}\omega,
\end{equation*}
where the average direction $\tilde{\Theta}$, expressed in the traveling wave profile, can be shifted to the origin without loss of generality due to the translation invariance of the PDE \eqref{eq:continuum_chimera_model_spatially_homogeneous}, i.e. $\tilde{\Theta}\equiv0$. Note that by Proposition~\ref{prop:traveling_wave_solution}, the solution in the traveling wave form recursively depends on the above polar order parameter magnitude $R$ as well as the traveling wave velocity $v$, i.e., $g(\omega) = g(\omega; R, v)$. Thus, the above equation can be considered as complex-valued self-consistent equation (SCE) for $R$ and $v$. Alternatively, by expanding it into the real and imaginary parts, we see that the order parameter must satisfy the following set of SCEs:
\begin{equation}
\label{eq:selfconsistent_equations}
	R = \int_\mathbb{T} g(\omega; R,v)\cos\omega\; \mathrm{d}\omega,\qquad 0 = \int_\mathbb{T} g(\omega; R,v)\sin\omega\; \mathrm{d}\omega.
\end{equation}
This system does not have an analytical solution but we can solve it numerically for polarization $R$ and additionally for velocity $v$ using the Newton-Raphson method for the system \eqref{eq:selfconsistent_equations}. The numerical solution of this system in terms of the model parameters $\alpha$ and $D_\varphi$ is presented in Figs.~\ref{fig:selfconsistent_equations_nonzero_alpha}(a) and (b).

\begin{figure}[]
	\centering
	\includegraphics[width=1.0\textwidth]{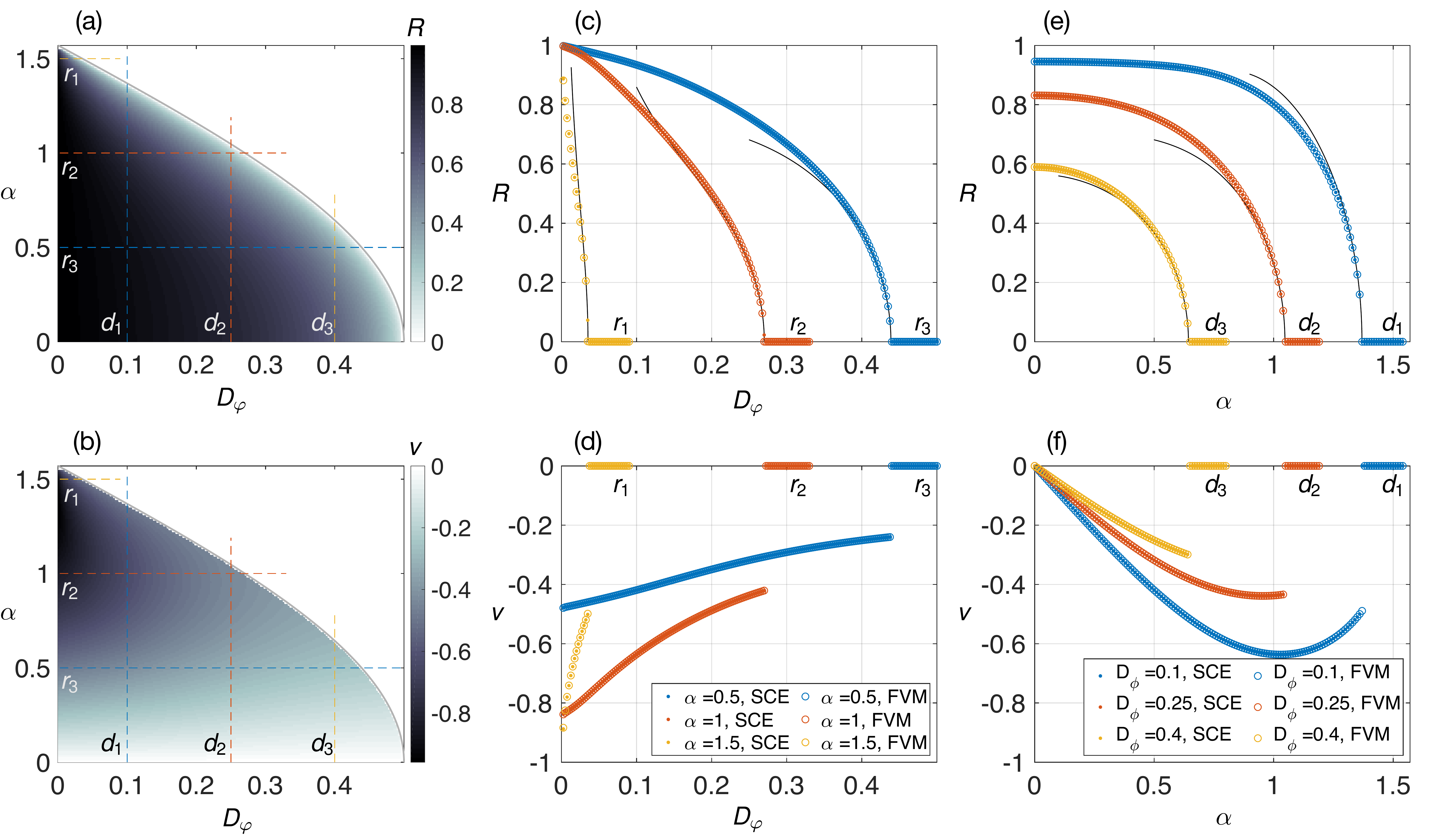}
	\caption
	{
		\label{fig:selfconsistent_equations_nonzero_alpha} 
		(a),(b) Solution of the system of SCEs \eqref{eq:selfconsistent_equations} comprising a density function as a traveling wave solution \eqref{eq:traveling_wave_solution} and a complex order parameter defined by that solution. The system determines (a) magnitude of the order parameter $R$ and (b) a group velocity $v$ of the traveling wave versus the phase lag $\alpha$ and diffusion coefficient $D_\varphi$. The gray line indicates the order-disorder transition line $D_\varphi=\frac{\sigma}{2}\cos\alpha$. The critical group velocity along that line is $v=-\frac{\sigma}{2}\sin\alpha$. The colored lines show intervals of parameter values used in (c)-(f). The marks $r_{1,2,3}$, and $d_{1,2,3}$ denote corresponding lines. (c),(d) Evolution of the order parameter magnitude $R$ and the group velocity $v$, respectively, versus the diffusion coefficient $D_\varphi$ for different phase lag values $\alpha$. (e),(f) Evolution of $R$ and $v$, respectively, versus $\alpha$ for different diffusion coefficients $D_\varphi$. The dots denote values found by solving the system of SCEs \eqref{eq:selfconsistent_equations} whereas the circles denote values produced by the FVM with the grid of $L=256$ points. The black lines show the exact form of $R$ next to the order-disorder transition line, as predicted by the hydrodynamic theory.
	}
\end{figure}

Given the above result, we are able to investigate phase transitions related to a spatially homogeneous rotating system \eqref{eq:spatially_nonhomogeneous_continuum_pde_with_nonzero_lag}. For each parameter set, we solve it starting from quasiuniform initial conditions, described earlier. We discretize the domain $\mathbb{T}$ into $L=256$ points and perform all computations of this section with the time step $\Delta t=10^{-4}$ until $t=10^4$. The results are presented in Figs.~\ref{fig:selfconsistent_equations_nonzero_alpha}(c) and (e) for the order parameter magnitude $R$, and in Figs.~\ref{fig:selfconsistent_equations_nonzero_alpha}(d) and (f) for the group velocity $v$. Because either the coupling strength $\sigma$ or the diffusion coefficient $D_\varphi$ may be eliminated by an appropriate rescaling of time in Eq.~\eqref{eq:spatially_nonhomogeneous_continuum_pde_with_nonzero_lag}, we fix $\sigma=1$ without loss of generality and investigate the behavior in terms of $\alpha$ and $D_\varphi$. Since a phase transition may happen by changing either the diffusion level $D_\varphi$ or the phase lag $\alpha$, we look into both possibilities. First, we fix $\alpha=0.5,1.0,1.5$ and investigate the behavior of $R=R(D_\varphi)$ (cf. Fig.~\ref{fig:selfconsistent_equations_nonzero_alpha}(c)) and $v=v(D_\varphi)$ (cf. Fig.~\ref{fig:selfconsistent_equations_nonzero_alpha}(d)). The phase transition occurs at $D_\varphi=\frac{\sigma}{2}\cos\alpha$, which one obtains by substituting the traveling wave solution \eqref{eq:traveling_wave_solution_profile} into SCEs \eqref{eq:selfconsistent_equations} and expanding right-hand sides with respect to $R$ around $R=0$ \cite{kruk2020pre}. The black line additionally shows how the order parameter behaves next to the transition line, according to
\begin{equation*}
	R \approx \sqrt{\frac{4D_\varphi^2 + v^2}{D_\varphi}(\cos\alpha-2D_\varphi)},
\end{equation*}
which is known from the hydrodynamic description of the particle model \eqref{eq:chimera_sde} and derived in \cite{kruk2020pre}. The grid size for $D_\varphi$ was chosen $0.0025$. For comparison reasons, we show the results obtained with the FVM and those of SCEs \eqref{eq:selfconsistent_equations}. Next, we fix $D_\varphi=0.1,0.25,0.4$ and investigate the behavior of $R=R(\alpha)$ (cf. Fig.~\ref{fig:selfconsistent_equations_nonzero_alpha}(e)) and $v=v(\alpha)$ (cf. Fig.~\ref{fig:selfconsistent_equations_nonzero_alpha}(f)). Again, as expected, the phase transition occurs at $\alpha=\arccos\left(\frac{2D_\varphi}{\sigma}\right)$. The grid size for $\alpha$ was chosen $0.01$. The black line and the results of the SCEs \eqref{eq:selfconsistent_equations} are obtained as before.

A few remarks on phase transitions between spatially homogeneous states are in order. We quantify them in terms of a polar order parameter magnitude $R$. Therefore, under the assumption of spatial homogeneity, we observe a second order transition from partially synchronized to disordered state (cf. Figs.~\ref{fig:selfconsistent_equations_nonzero_alpha}(c) and (e)). However, it is known that for polar active matter systems, the transition from uniform disordered to ordered motion is separated by a region with density-segregated patterns such as traveling bands. It has been established that for the Boltzmann equation for self-propelled particle systems, the respective transitions may be of first and second order depending on a system size \cite{thueroff:prx2014}. As we mentioned earlier, we do not observe such a spatially inhomogeneous regime close to the order-disorder transition line because nonlocal alignment interactions in \eqref{eq:continuum_chimera_model} are normalized, which makes spatially homogeneous solutions more stable against spatially inhomogeneous perturbations \cite{kruk2020pre}.

\subsection{Spatially inhomogeneous solutions\label{sec:spatially_nonhomogeneous_solutions}}

\begin{figure}[t]
	\centering
	\includegraphics[width=0.85\textwidth]{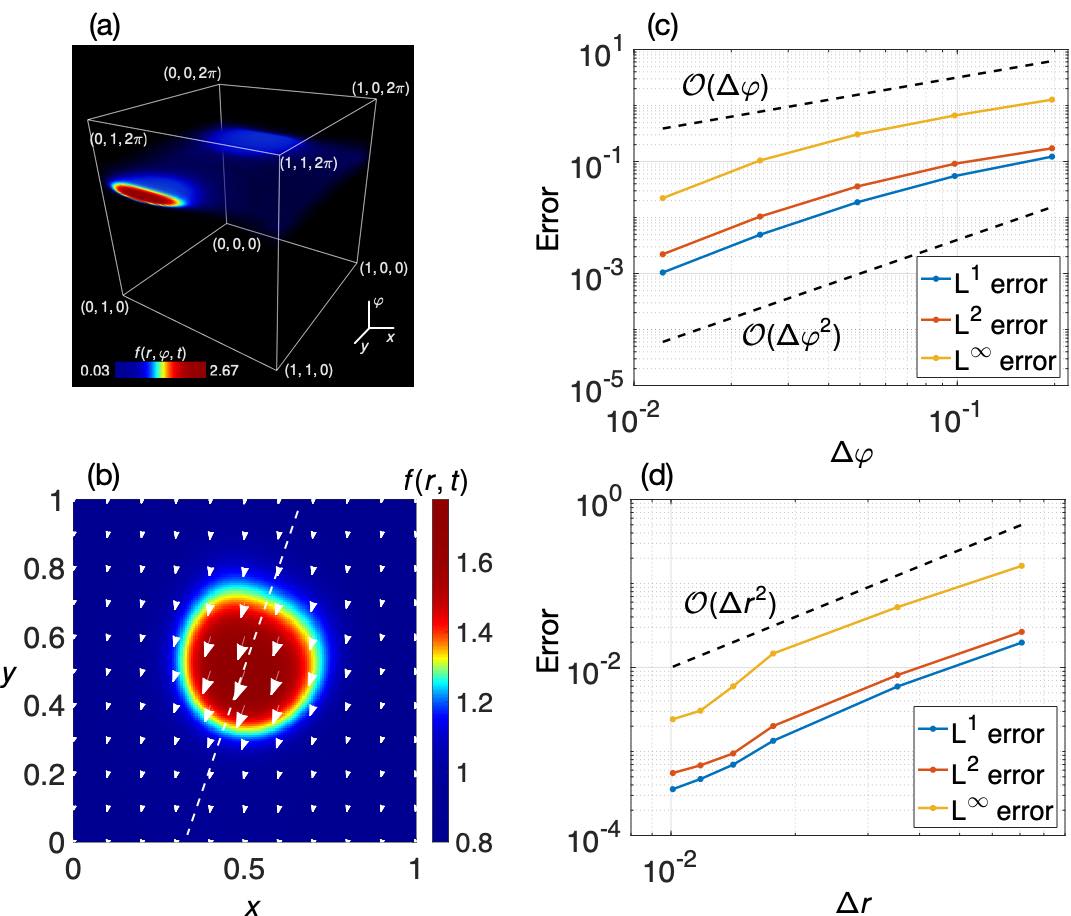}
	\caption
	{
		\label{fig:test_case_localized_chimera}
		(a) Spatially inhomogeneous solution of \eqref{eq:spatially_nonhomogeneous_continuum_pde_with_nonzero_lag} in the form of localized chimera state. The grid contains $N\times M\times L=160\times160\times128$ points. 
		(b) Its projection \eqref{eq:hydrodynamic_spatial_density} into spatial coordinates. White arrows indicate the momentum field \eqref{eq:hydrodynamic_momentum_field} generated by the solution. The dashed white line \eqref{eq:parameterized_line} passes through the center of maximal density of the projection and is aligned with the momentum field at this point.
		(c,d) Convergence of error for the same problem and parameter region in $L^1$, $L^2$, and $L^\infty$ norms at $t=1$. The points correspond to grids (c) in $\varphi$ with 32, 64, 128, 256, and 512 points, and (d) in $x,y$ with $20\times20$, $40\times40$, $80\times80$, $100\times100$, $120\times120$, and $140\times140$ points. Here $\Delta r=\sqrt{\Delta x^2 + \Delta y^2}$. The grid also contains (c) $40\times40$ points in $x,y$, (d) 128 points in $\varphi$. Other parameters are $v_0=1$, $\sigma=4$, $\varrho=0.3$, $\alpha=1.54$, $D_\varphi=0.01$.
	}
\end{figure}

We have seen that the presented finite volume scheme is capable of reproducing correct behavior under the assumption of spatial homogeneity. Therefore, we proceed to the general PDE \eqref{eq:continuum_chimera_model} in a parameter region where spatially inhomogeneous patterns occur. We consider the same IVP \eqref{eq:spatially_nonhomogeneous_continuum_pde_with_nonzero_lag} from the previous section. From the linear stability analysis of this PDE from the point of view of kinetic theory \cite{kruk2020pre}, we know that there exists a region in the parameter space of $\varrho/v_0$, $\alpha$, and $D_\varphi$ where the spatially homogeneous skewed unimodal density function \eqref{eq:traveling_wave_solution_profile} becomes unstable against spatially dependent perturbations and one observes numerous spatially inhomogeneous patterns.

In this paper, we would like to concentrate on the analysis of one of such patterns, which we refer to as a localized chimera state \cite{kruk:aps2018}. From the point of view of the particle model \eqref{eq:chimera_sde}, such a state simultaneously consists of two rather distinct groups of particles. The first group forms a subset of particles that gather into a circular polarized cloud, which rotates in the background of the rest of chaotically moving particles. In the continuum limit, this state is characterized by the formation of a high density skewed ellipsoidal region in $\Omega$ (cf. Fig.~\ref{fig:test_case_localized_chimera}(a)) that follows a helical path. We can obtain the aforementioned high density circular cloud as a projection of such a solution into spatial coordinates (cf. Fig.~\ref{fig:test_case_localized_chimera}(b))
\begin{equation}
\label{eq:hydrodynamic_spatial_density}
	f(r,t) = \int_\mathbb{T} f(r,\varphi,t)\; \mathrm{d}\varphi.
\end{equation}
The spatial profile has a form of a bivariate solitary wave which possesses a characteristic front-end asymmetry. This can be observed if we look at the solution profile along the line (cf. Fig.~\ref{fig:nonhomogeneous_phase_transitions}(b), a white dashed line), centered at the point of maximal density and directed according to the velocity field at this point. We look for the point of maximal spatial density as
\begin{equation*}
	r_\text{max} = \argmax_{r\in\mathbb{U}^2} f(r,t).
\end{equation*}
The velocity field can be retrieved from the momentum field, in turn, obtained as
\begin{equation}
\label{eq:hydrodynamic_momentum_field}
	u(r,t) = \int_\mathbb{T} e(\varphi)f(r,\varphi,t)\; \mathrm{d}\varphi
\end{equation}
where $e(\varphi)=(\cos\varphi,\sin\varphi)\in\mathbb{S}^1$. Note that the momentum field such defied is isomorphic to the global polar order parameter, we used earlier \eqref{eq:global_polar_order_parameter}. As a result, the line can be parameterized as
\begin{equation}
\label{eq:parameterized_line}
	r(s) = r_\text{max} + e(\varphi_\text{max}) s,\quad s\in\left[-\half,\half\right],
\end{equation}
where $\varphi_\text{max} = \arg(u(r_\text{max}))$. Using the piecewise linear reconstruction of the density function \eqref{eq:piecewise_linear_reconstruction_3d}, we can find the approximate values of the density function at any point on the line (cf. Fig.~\ref{fig:nonhomogeneous_phase_transitions}(b)). Moreover, for solutions, where the radius of rotation of the localized cluster is sufficiently large so that the cluster does not rotate around a fixed point, the transverse profile of the spatial projection $f(r,t)$ has a symmetric form. 

We remark that Eqs.~\eqref{eq:hydrodynamic_spatial_density},\eqref{eq:hydrodynamic_momentum_field} constitute a hydrodynamic description of the kinetic PDE \eqref{eq:continuum_chimera_model} where polar order is expected to emerge, and is often used to get analytical insights into the dynamics. However, the known drawback of this approach is that it is limited to the regimes close to equilibrium. As opposed to that, our FVMs are applicable to any region in the parameter space, which we will employ in the following.

In a general spatially inhomogeneous setup, we are able to calculate error convergence in both angular and spatial variables. Because we do not know any exact solution with spatial dependence, we use Eqs.~\eqref{eq:error_norms_benchmark} to compute the norms. The reference solution was the one with $N\times M\times L=40\times40\times1024$ grid points. First, we fix $N=40$, $M=40$ and vary the angular grid size $\Delta\varphi$. We observe the second order convergence for $\Delta\varphi$ sufficiently small whereas it approaches the first order for largest grid sizes (cf. Fig.~\ref{fig:test_case_localized_chimera}(c)). The reason for the first order behavior is that such discretizations cannot capture high density gradients in $\varphi$ so that the numerical error is accumulated rather fast. Subsequently, since the dynamics in angular and spatial dimensions are coupled, this results in solutions, diffused away from the correct dynamics. The time step was chosen $\Delta t=5\cdot10^{-4}$. Next, we fix $L=128$ and vary the spatial grid size $\Delta x,\Delta y$. We again observe the second order error convergence versus $\Delta r=\sqrt{(\Delta x)^2 + (\Delta y)^2}$ even for quite small grid sizes with $N\times M=20\times20$ points (cf. Fig.~\ref{fig:test_case_localized_chimera}(d)). We explain such robustness of the results by the fact that the solutions were computed for quite large $\varrho=0.3$, resulting in long range interactions. The time step was chosen $\Delta t=5\cdot10^{-4}$. The reference solution was the one with $N\times M\times L=160\times160\times128$ grid points.

\subsection{Phase transitions of spatially inhomogeneous solutions\label{sec:phase_transitions_of_spatially_nonhomogeneous_solutions}}

In Section~\ref{sec:phase_transitions_of_spatially_homogeneous_solutions}, we found that in spatially homogeneous systems, the transition between polar order and disorder is of second order. Now that we have an established protocol to generate spatially inhomogeneous solutions, we are interested to learn about their related phase transitions. By choosing an appropriate scale for the microscopic particle velocity $v_0$ and interparticle interaction radius $\varrho$, we find a phase diagram in the $(\alpha,D_\varphi)$-parameter space where three distinct solutions are observed, i.e., a spatially homogeneous disordered motion (SHDM), a spatially homogeneous ordered motion (SHOM), and a spatially inhomogeneous motion (SNM) in the form of a localized chimera state (cf. Fig.~\ref{fig:nonhomogeneous_phase_transitions}(a)). The figure was obtained from the kinetic linear stability analysis of Eq.~\eqref{eq:traveling_wave_solution_profile}, performed in detail in \cite{kruk2020pre}. We fix $v_0=0.25$, $\sigma=1$, $\varrho=0.3$. Phase transitions discussed earlier correspond to the transition between SHDM and SHOM, with the bifurcation occurring at $D_\varphi=\frac{\sigma}{2}\cos\alpha$ (a black solid line). In the gray region, starting from quasirandom initial conditions, one observes the formation of spatial structures. The formation happens in two stages. First, the system smooths out initial spatial perturbations but gradually polarizes until the occurrence of a skewed angular profile similar to the one in Fig.~\ref{fig:test_case_nonstationary_phase_synchronization_1d}(a). Second, the remaining spatial perturbations start to act on the solution and accumulate eventually into a bivariate unimodal shape. We remark that due to finite numerical precision, the second step might not be triggered for any quasirandom initial conditions given that spatial variations become of order of round-off error $\mathcal{O}(10^{-18})$. To circumvent that, one might either look for such initial conditions that preserve enough spatial perturbations at the time point of maximal synchronization or use multiprecision arithmetic libraries, like the one we used in our implementation.

\begin{figure}[t]
	\centering
	\includegraphics[width=1.0\textwidth]{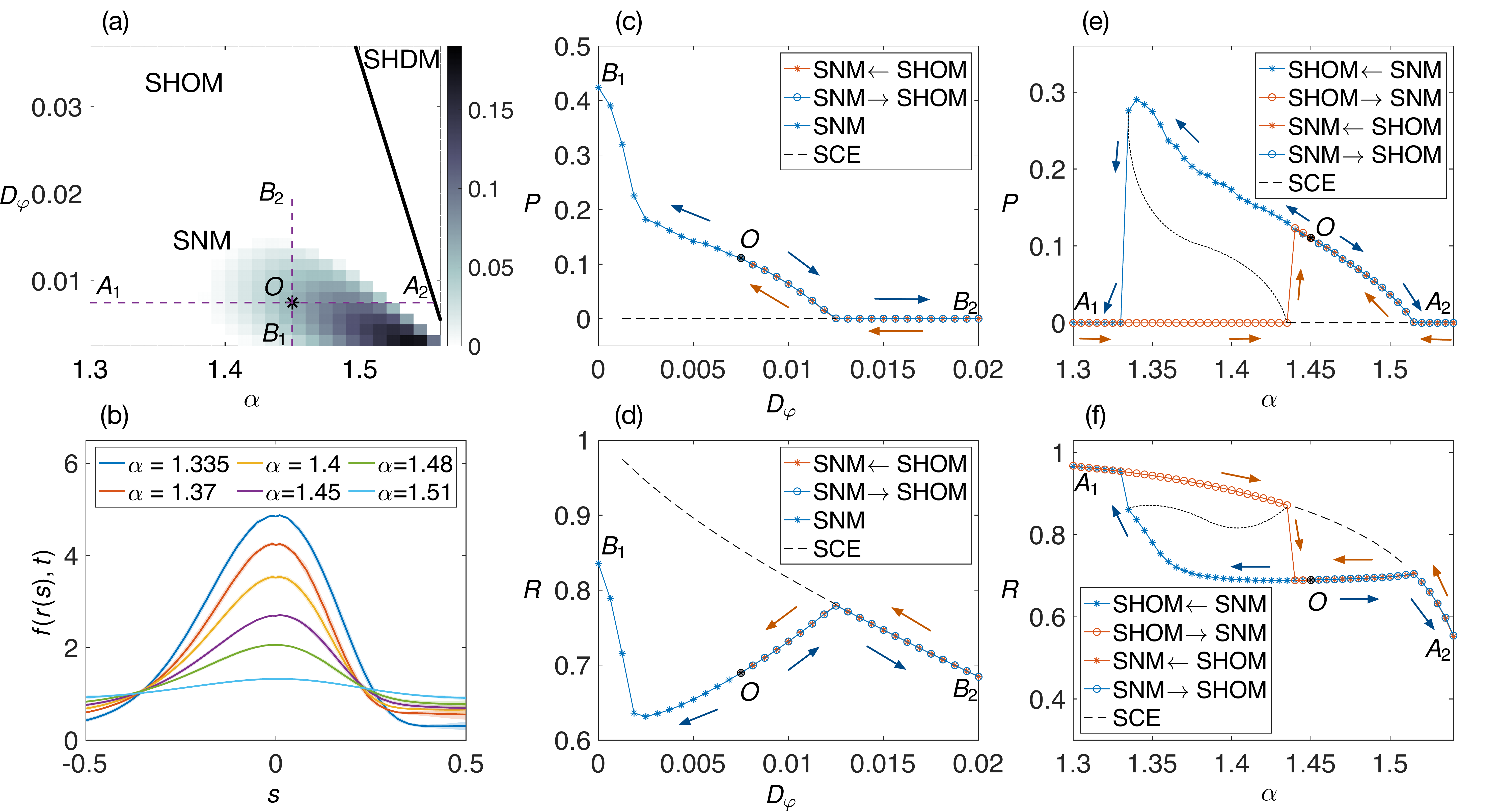}
	\caption
	{
		\label{fig:nonhomogeneous_phase_transitions}
		Bifurcations scenarios for SNM, represented by localized chimera states. (a) Phase diagram in the parameter space of diffusion $D_\varphi$ and phase lag $\alpha$; color shows the maximal real part of the strongest unstable Fourier mode \cite{kruk2020pre}; the black solid line is the order-disorder transition line $D_\varphi=\frac{\sigma}{2}\cos\alpha$. Black star indicates the initial value $O=(\alpha,D_\varphi)=(1.45,0.0075)$ used as a starting point for the continuation method (see the text). Purple dashed lines show continuation paths presented in (c,d,e,f).
		(b) Spatial profiles along the line of nonlocal collective motion in Fig.~\ref{fig:test_case_localized_chimera}(b) for different values of the phase lag taken along the branch of stability of SNM in (e). Solid lines represent averaged densities over 10 time units and shaded regions denote respective standard deviations.
		(c,d) Localization $P$ \eqref{eq:localization_order_parameter} and polar $R$ \eqref{eq:global_polar_order_parameter} order parameters versus $D_\varphi$, respectively. One observes a second order phase transition between SNM and SHOM with bifurcation at $D_\varphi\approx0.0125$. Black dashed lines denote unstable branches and have been computed from the system of SCEs \eqref{eq:selfconsistent_equations}. (e,f) Localization $P$ and polar $R$ order parameters versus $\alpha$, respectively. One observes two types of phase transitions, namely, the first order one, accompanied with a hysteresis loop, on the path $OA_1$ with bifurcation points $\alpha\approx1.33,1.44$ and the second order one on the path $OA_2$ with bifurcation at $\alpha\approx1.515$. Black dashed lines denote unstable branches and have been obtained from \eqref{eq:selfconsistent_equations}. Black dotted lines are drawn "by hand" in place of unknown unstable branches. Colored arrows in (c,d,e,f) indicate directions of bifurcation paths.
	}
\end{figure}

As one can see in Fig.~\ref{fig:nonhomogeneous_phase_transitions}(a), the new phase transitions should occur between SHOM and SNM by varying either the diffusion coefficient $D_\varphi$ or the phase lag parameter $\alpha$. We inspect each route separately. Before we do that, we need to establish an appropriate order parameter to measure the level of spatial localization induced by a PDE solution as well as be able to detect changes in spatial variation of solutions upon varying model parameters. First, in a similar way as we might consider the global polar order parameter $R(t)$ \eqref{eq:global_polar_order_parameter} as a measure of angular localization of orientation vectors $e(\varphi)$ belonging to $\mathbb{S}^1$, manifested in a momentum field definition \eqref{eq:hydrodynamic_momentum_field}, we define an order parameter that measures the level of spatial localization of elements belonging to $\mathbb{S}^1\times\mathbb{S}^1$ in the following way:
\begin{equation}
\label{eq:localization_order_parameter}
	P(t)e^{i\Psi(t)} = \int_\Omega f(r,\varphi,t)e^{i2\pi(x+y)}\; \mathrm{d}r\mathrm{d}\varphi.
\end{equation}
This parameter provides the following information. For systems with all the probability mass compressed in one point, i.e., for point measures, the spatial localization is most pronounced and the magnitude of the order parameter attains its maximal value $P=1$. In the opposite case, for systems with uniform distribution of matter, no spatial localization is observed and the order parameter magnitude attains its minimal value $P=0$. For partial localization inside a particle flow, we therefore have $P\in(0,1)$. The phase $\Psi$ is irrelevant to our purposes.
Second, to detect changes in spatial structure of solutions while changing model parameters, we introduce the following maximum absolute spatial deviation measure \cite{thueroff:prx2014}:
\begin{equation}
\label{eq:maximum_absolute_spatial_deviation}
	\delta_r(t) = \max_{(i,j,k)\in\Omega_{N,M,L}} \left\{\left\vert f_{i,j,k}(t) - f_{k}(t) \right\vert\right\},
\end{equation}
where spatially averaged solutions are computed as
\begin{equation*}
	f_{k}(t) = \frac{1}{NM} \sum_{(i,j)\in\mathbb{U}_N\times\mathbb{U}_M} f_{i,j,k}(t),\quad k\in\mathbb{T}_L.
\end{equation*}
For SHOM, this measure attains values of order $\mathcal{O}(10^{-14})$, when the magnitude of spatial variations is of a round-off error for double precision floating point values.

We now describe the transitions between SNM and SHOM. We start with a parameter point well inside a region where SNM is a stable solution, i.e. $O=(\alpha,D_\varphi)=(1.45,0.0075)$ (cf. Fig.~\ref{fig:sequence_of_spatially_nonhiomogeneous_solutions}(b)), and proceed in a continuation-like manner. First, we fix the phase lag parameter $\alpha$ and increase the diffusion level $D_\varphi\in[0.0075,0.02]$ with a parameter step size $\Delta D_\varphi=0.000625$ (cf. Fig.~\ref{fig:nonhomogeneous_phase_transitions}(a), a vertical path $OB_2$). Starting from quasirandom initial conditions \eqref{eq:spatially_nonhomogeneous_continuum_pde_with_nonzero_lag}, we let the system to converge to a solitary wave form of a localized chimera state and take this solution as an initial condition for a subsequent computation. Then, we change the diffusion level, take as a new initial condition the final solution from a previous parameter, and let the system equilibrate for $T=100$ time units with $\Delta t=5\cdot10^{-3}$. Afterwards, we accumulate values of $\delta_r(t)$ \eqref{eq:maximum_absolute_spatial_deviation} and continue integration until spatial deviations cease to fluctuate with $\mathrm{d}\delta_r(t)/\mathrm{d}t < 5\cdot10^{-5}$. We quantified the rate of change $\mathrm{d}\delta_r(t)/\mathrm{d}t$ as a linear fit to a set of values $\delta_r(t)$ over the last $50$ time units.
The result of this procedure is shown in Fig.~\ref{fig:nonhomogeneous_phase_transitions}(c). As we see, the spatial order parameter $P$ assumes a continuous path versus the diffusion constant $D_\varphi$. We then start from a point $B_2=(1.45,0.02)$ and go in the reverse direction gradually decreasing $D_\varphi$ with the same step size as before. For each new parameter, we take as an initial condition the final state of the system from a previous parameter and impose small spatial perturbations of the same form as in \eqref{eq:spatially_nonhomogeneous_continuum_pde_with_nonzero_lag} in order to allow spatial perturbations to grow provided that SHOM is unstable. We note that for spatially homogeneous systems, spatial variations are of order of a round-off error, and without such initial spatial perturbations, spatial deviations never grow even for parameter values where SHOM is indeed unstable. During this reverse round of continuation simulations, we integrate the system until $\mathrm{d}\delta_r(t)/\mathrm{d}t < 10^{-6}$. As a result, by varying the diffusion constant $D_\varphi$, we observe a supercritical transition between SNM and SHOM on the path $OB_2$ at $D_\varphi\approx0.0125$. Additionally, we provide the results of the continuation procedure in terms of the polar order parameter $R$ \eqref{eq:global_polar_order_parameter} (cf. Fig.~\ref{fig:nonhomogeneous_phase_transitions}(d)) to make the comparison with SHOM transitions.


As the next step, we study phase transitions between SNM and SHOM versus the phase lag $\alpha$. There are two ways, they can occur. To begin with, let us follow the right path $OA_2$ in Fig.~\ref{fig:nonhomogeneous_phase_transitions}(a). As an initial parameter point, we again set $O=(1.45,0.0075)$, keep $D_\varphi$ constant, and vary $\alpha\in[1.45,1.54]$ with a parameter step size $\Delta\alpha=0.005$. We follow the same continuation protocol as before and report the results in Fig.~\ref{fig:nonhomogeneous_phase_transitions}(e,f). We see that the transition between SNM and SHOM versus $\alpha$ on this path follows the similar scenario as the previously described transition along $OB_2$. That is, it is of second order with the bifurcation point $\alpha\approx1.515$. This comes as no surprise as both bifurcations occur close to the order-disorder transition line $D_\varphi=\frac{\sigma}{2}\cos\alpha$, where the effect of diffusion is substantial. Here, the increase of $D_\varphi$ is qualitatively similar to the decrease of $\alpha$. By doing so, a spatially localized region gradually smooths around.

\begin{figure}[]
	\centering
	\includegraphics[width=1.0\textwidth]{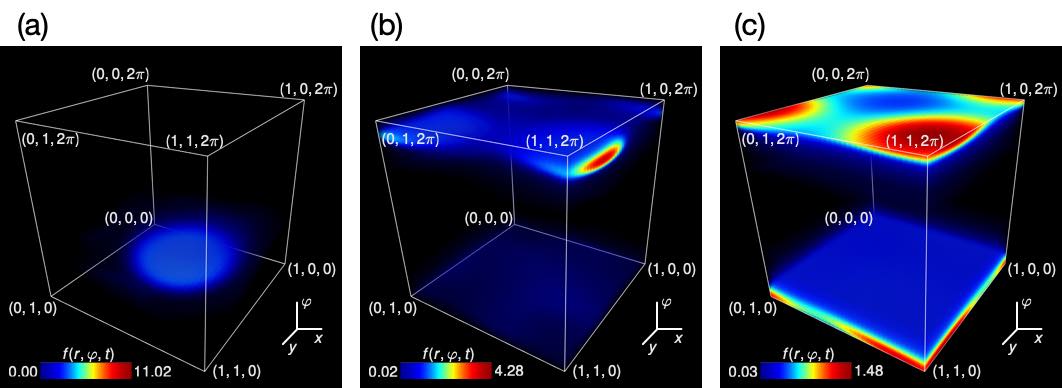}
	\caption
	{
		\label{fig:sequence_of_spatially_nonhiomogeneous_solutions}
		Evolution of SNM, represented by localized chimera states, along the bifurcation path $A_1A_2$ (cf. Fig.~\ref{fig:nonhomogeneous_phase_transitions}(a)). Snapshots of solutions (a) next to the first order transition at $\alpha=1.335$, (b) in the middle of the path at $\alpha=1.45$, and (c) before the second order transition at $\alpha=1.51$. The respective videos demonstrating temporal evolution of density functions can be found in \cite{bcs_youtube_channel,figshare}. Other parameters are $v_0=0.25$, $\sigma=1$, $\varrho=0.3$, $D_\varphi=0.0075$, $N=40$, $M=40$, $L=256$.
	}
\end{figure}

The second way, the transition between SNM and SHOM occurs, is along the left path $OA_1$ in Fig.~\ref{fig:nonhomogeneous_phase_transitions}(a). By performing the same continuation procedure for the parameter values $\alpha\in[1.3,1.45]$, we observe a hysteresis loop (cf. Figs.~\ref{fig:nonhomogeneous_phase_transitions}(e,f)), characteristic to first order transitions. Along the SNM$\rightarrow$SHOM path, we come across a bifurcation point $\alpha\approx1.33$ of a saddle-node type. Along the SHOM$\rightarrow$SNM path, we find a bifurcation point $\alpha\approx1.44$ of a subcritical type. Apart from the results of the linear stability analysis, which showed us where SNM is observable starting from any initial conditions (except for unstable solutions), we discover the existence of a bistability region where both SHOM and SNM are stable solutions, i.e., $\alpha\in(1.33,1.44)$ with $D_\varphi=0.0075$. Moreover, we observe some discrepancy between the results of the linear stability analysis and the continuation method. According to the stability analysis, starting from $\alpha\approx1.38$, SHOM should become unstable against spatially dependent perturbations while the continuation method provides $\alpha\approx1.44$. This is because the stability analysis was performed under the assumption of small microscopic particle velocities $v_0$ in a region of small diffusion, which is not the case here. Therefore, the numerical analysis of the PDE provides us with a better understanding of how solutions behave far from the order-disorder transition line.

Along bifurcation paths with respect to both parameters, SNM undergoes qualitatively similar transformations. Starting from second order transition points ($D_\varphi\approx0.0125$ in Fig.~\ref{fig:nonhomogeneous_phase_transitions}(c) and $\alpha\approx1.515$ in Fig.~\ref{fig:nonhomogeneous_phase_transitions}(e)), an ellipsoidal shape forms inside a high density layer (cf. Fig.~\ref{fig:sequence_of_spatially_nonhiomogeneous_solutions}(c)) but the layer itself does not disappear completely. In Fig.~\ref{fig:sequence_of_spatially_nonhiomogeneous_solutions}(b), one can observe coexistence of a localized cluster with such a layer for a parameter point even in the middle of a bifurcation path.
By decreasing parameters to minimal values with SNM being stable, the localized cluster is most clearly pronounced and the secondary layer dissolves (cf. Fig.~\ref{fig:sequence_of_spatially_nonhiomogeneous_solutions}(a)). The exemplary videos demonstrating temporal evolution of SNM solutions for different parameter values can be found in \cite{bcs_youtube_channel,figshare}. Fig.~\ref{fig:sequence_of_spatially_nonhiomogeneous_solutions} demonstrates qualitative changes in SNM with respect to the phase lag $\alpha$ but one obtains similar results by decreasing the diffusion level $D_\varphi$ towards zero.

\section{Conclusion\label{sec:conclusion}}

In this paper, we have discussed the problem of modeling systems of infinitely large populations of nonlocally interacting active Brownian particles. We have developed finite volume schemes to solve a class of nonlinear Vlasov-Fokker-Planck equations obtained in the continuum limit of such systems. According to the continuum limit methodology, these PDEs govern a temporal evolution of nonnegative probability density functions. Taking that into account, we considered the application of positivity-preserving slope limiters and the SSP-RK time discretization in our schemes in order to preserve the probabilistic nature of solutions. Because the problems of interest describe motion of isolated particle flows, the schemes additionally guarantee the conservation of probability mass. Given the theoretical insights on the system properties, we have considered separately one-dimensional problems for spatially homogeneous systems and general three-dimensional problems for spatially inhomogeneous ones. For each case, we have demonstrated that our finite volume schemes yield the second order with respect to discretization procedures. We have demonstrated that the schemes correctly reproduce known stationary and traveling wave solutions. In addition to the presentation of various continuum limit dynamics, we have performed the analysis on phase transitions between three classes of solutions, i.e., SHDM, SHOM, and SNM. This analysis has revealed the existence of both first and second order transitions with respect to changes in either the diffusion level or the phase lag.

\section*{Acknowledgments}

JAC was partially supported by EPSRC grant number EP/P031587/1 and the Advanced Grant Nonlocal-CPD (Nonlocal PDEs for Complex Particle Dynamics: Phase Transitions, Patterns and Synchronization) of the European Research Council Executive Agency (ERC) under the European Union's Horizon 2020 research and innovation programme (grant agreement No. 883363). HK acknowledges support from the European Research Council (ERC) with the consolidator grant CONSYN (Grant No. 773196).

\bibliography{finite_volume_schemes}

\end{document}